\documentclass{imsart}
\RequirePackage{natbib}
\usepackage{amsmath,amssymb,amsthm}
\usepackage{graphics,epsfig}
\usepackage{hyperref}
\usepackage{natbib}
\usepackage{color}
\usepackage{graphicx}
\usepackage{caption}
\usepackage{subcaption}
\usepackage{float}
\usepackage{dsfont}
\usepackage{mathrsfs}
\usepackage{multirow}
\usepackage{comment}
\usepackage{bm}

\def \qq {\boldsymbol{q}}

\def \ra {\rightarrow}

\def \E {\mathbb{E}}

\newtheorem{definition}{\bf Definition}

	\newtheorem{theorem}{\bf Theorem}
	
	\newtheorem{prop}{\bf Proposition}
	
	\newtheorem{as}{\bf Assumption}
	
	\newtheorem{cor}[theorem]{\bf Corollary}

\renewcommand{\epsilon}{\varepsilon}

\begin{document}

\begin{frontmatter}

\title{Consensus as a Nash Equilibrium of a stochastic differential game}
\runtitle{Dynamic  Network}

\begin{aug}

\author{\fnms{Paramahansa} 
	\snm{Pramanik}
	\ead[label=e1]{ppramanik@southalabama.edu}}

\runauthor{P. Pramanik}

\affiliation{University of South Alabama}

\address{Department of Mathematics and Statistics\\ University of South Alabama\\ Mobile, AL 36688 USA.}

\end{aug}
  
\begin{abstract}
In this paper a consensus has been constructed in a social network which is modeled by a stochastic differential game played by agents of that network. Each agent independently minimizes a cost function which represents their motives. A conditionally expected integral cost function has been considered under an agent's opinion filtration. The dynamic cost functional is minimized subject to a stochastic differential opinion dynamics. As opinion dynamics represents an agent's differences of opinion from the others as well as from their previous opinions, random influences and stubbornness make it more volatile. An agent uses their rate of change of opinion at certain time point as a control input. This turns out to be a non-cooperative stochastic differential game which have a feedback Nash equilibrium. A Feynman-type path integral approach has been used to determine an optimal feedback opinion and control. This is a new approach in this literature. Later in this paper an explicit solution of a feedback Nash equilibrium opinion is determined.
\end{abstract}



\end{frontmatter}

\section{Introduction}
Social networks influence a lot of behavioral activities including educational achievements \citep{calvo2009}, employment \citep{calvo2004}, technology adoption \citep{conley2010}, consumption \citep{moretti2011} and smoking \citep{nakajima2007,sheng2020}. As social networks are the result of individual decisions, consensus takes an important role to understand the formation of networks. Although a lot of theoretical works on social networks have been done \citep{jackson2010,goyal2012,sheng2020}, work on consensus as a Nash equilibrium under a stochastic network is very insignificant \citep{niazi2016}. \cite{sheng2020} formalizes network as simultaneous-move game, where social links based on decisions are based on utility externalities from indirect friends. \cite{sheng2020} proposes a computationally feasible partial identification approach for large social networks. The statistical analysis of network formation dates back to the seminal work by \cite{erdos1959} where a random graph is based on independent links with a fixed probability \citep{sheng2020}. Beyond Erd\"os-R\'enyi model, many methods have been designed to simulate graphs with  characteristics like degree distributions, small world, and Markov type properties \citep{polansky2021}. A model based method is useful if this model can be fit successfully and if it is a relatively simple to simulate realizations \citep{pramanik2016tail,polansky2021}. The most frequently used general model for random graphs is the exponential random graph model (ERGM) \citep{snijders2002,hua2019,polansky2021} because, this model fits well with the observed network statistics \citep{sheng2020}. This ERGM model lacks microfoundations which are important for counterfactual analyses and furthermore, economists view network analysis as the optimal choices of agents who maximizes their utilities \citep{sheng2020}. Network evolves as the result of a stochastic process is another popular framework where network may be observed, but it is the parameters of the stochastic process that are of interest, and the observed network is a single realization of the stochastic process \citep{polansky2021}.

At birth, humans already posses different types of skills like breathing, digest foods and motor actions which make a human body to behave like an automaton \citep{kappen2007}. Furthermore, like other animals humans acquire skills through learning. Different person has different abilities to acquire a new information in order to get an idea about pleasure, danger or food \citep{kappen2007}. Humans are the most complicated species on earth because, their decisions are not linear and they can learn difficult skills through transitional signals of their complex constellations of sensory patterns. For example, if food is kept in front of a hamster, it would eat immediately. On the other hand, if a plate of food is kept in front of a human, they might not eat because, variety of factors such as the texture, smell, amount of it, their sociocultural background, religion and ethnicity take place before even they think about to tastes it. In order to make this decision, a lot of complex neural activities take place inside a person's brain. Action of two main parts of a human brain, frontal and occipital lobes, makes them decide what they should do after seeing an object. In this case the occipital lobe sends information of an object through the synaptic systems to frontal lobe, which is known by their previous experiences and knowledge. As for humans one has to consider so many other possibilities compared to a hamster, such that they can choose any of the all available information with some probabilities and make decisions based on it.

This type of human behavior is a feedback circuit where the learning algorithm is determined by a synaptic motor command, more time with an object not only leads to get more information but also the knowledge to adapt with it in the long run and get more intelligence. For example, as humans grow older, more intelligent they become and reflects their genotype closely. On the other hand, environment influences a certain type of decision more with older ages which comes through a process called Hebbian learning \citep{hebb2005,kappen2007}. Ancestors gather more information about an object or circumstance and transfer it to their off-springs in order to help them survive easily and make decisions rationally \cite{kappen2007}. For example, without having a prior knowledge one does not know how to get a certain type restaurant and which lead them explore their surroundings. If that person finds out a restaurant, they survive for that day. On the next day, they might not be completely sure about full availability of food in that restaurant because of sudden environmental degradation after his previous visit such as flash flood, tornado, an avalanche, landslide, earthquake, other activities like closure due to burglary, fire or some gun related activities so on. Even if that person is sure about the availability of food, they might not go because of other socioeconomic behaviors at the back of their mind. Hence, more information might not lead them react rationally. These types of activities occurs when an event is more uncertain. Consider person $A$ is selling their $1$ million-dollar car to another person $B$ by just $\$ 50,000$. The rationality assumption suggests person $B$ to go for this offer but, $B$ might think why $A$ is giving this offer and might be suspicious about the quality of that car and rejects it.

Therefore, subjective probabilities take an important role to make these types of decisions based on individual judgments such as success in a new job, outcome of an election, state of an economy, difference in learning a new complex topic among students, spreading gossips in small communities  \citep{kahneman1972,niazi2016,tversky1971}. People follow representativeness in judging the likelihood of uncertain events where the probability is defined by the similarities in essential properties to its parent population and reflect the salient features of the process by which it is generated \citep{kahneman1972}, which makes opinion dynamics of a person to follow a stochastic differential equation. Furthermore, an individual minimizes its cost of foraging for food where finding food can be termed as a reward to them and they want to find their reward with minimal cost. Assume an agent discounts more to the recent future than farther future represented as feedback motor control reinforcement learning problem \cite{kappen2007}. In an environment of very complex opinion dynamics each agent minimizes their integral cost function subject to a stochastic differential opinion dynamics based on all above cases. This paper considers two environments first, all the agents have same opinion power and second, agents with a leader, where the leader has more power in opinion than others and determines their opinion first based on their own cost minimization mechanism. A feedback Nash equilibrium of opinion is determined by a Feynman-type path integral approach which so far from my knowledge is new \citep{feynman1949,pramanik2021opt}. Furthermore, this approach can be used to obtain a solution for stability of an economy after pandemic crisis \citep{ahamed2021}, determine an optimal bank profitability \citep{hossain2015,ahamed2021d}.

As each agent's opinion in a society is assumed to be a quantum particle, I introduce an alternative method based on Feynman-type path integral to solve this stochastic opinion dynamics problem based on Feynman-type path integrals instead of traditional Pontryagin Maximum Principle. If the objective function is quadratic and the differential equations are linear, then solution is given in terms of a number of Ricatti equations which can be solved efficiently \citep{kappen2007a}. But the opinion dynamics is more complicated than just an ordinary linear differential equation and non-linear stochastic feature gives the optimal solution a weighted mixture of suboptimal solutions, unlikely in the cases of deterministic or linear optimal control where a unique global optimal solution exists \citep{kappen2007a}. In the presence of Wiener noise, Pontryagin Maximum Principle, a variational principle, that leads to a coupled system of stochastic differential equations with initial and terminal conditions, gives a generalized solution \citep{kappen2007a,oksendal2019}. Although incorporate randomness with its Hamiltonian-Jacobi-Bellman (HJB) equation is straight forward but difficulties come due to dimensionality when a numerical solution is calculated for both of deterministic or stochastic HJB \citep{kappen2007a}. General stochastic control problem is intractable to solve computationally as it requires an exponential amount of memory and computational time because, the state space needs to be discretized and hence, becomes exponentially large in the number of dimensions \citep{theodorou2010,theodorou2011,yang2014path}. Therefore, in order to calculate the expected values it is necessary to visit all states which leads to the summations of exponentially large sums \citep{kappen2007a,yang2014path}. \cite{kappen2005linear} and \cite{kappen2005} say that a class of continuous non-linear stochastic finite time horizon control problems can be solved more efficiently than Pontryagin's Maximum Principle. These control problems reduce to computation of path integrals interpreted as free energy because, of their various statistical mechanics forms such as Laplace approximations, Monte Carlo sampling, mean field approximations or belief propagation \citep{kappen2005linear,kappen2005,kappen2007a,van2008}. According to \cite{kappen2007a} these approximate computations are really fast.

Furthermore, one can transform a class of non-linear HJB equations into linear equations by doing a logarithmic transformation. This transformation stems back to the early days of quantum mechanics which was first used by Schr\"odinger to relate HJB equation to the Schrödinger equation \citep{kappen2007a}. Because of this linear feature, backward integration of HJB equation over time can be replaced by computing expectation values under a forward diffusion process which requires a stochastic integration over trajectories that can be described by a path integral \citep{kappen2007a}. Furthermore, in more generalized case like Merton-Garman-Hamiltonian system, getting a solution through Pontryagin Maximum principle is impossible and Feynman path integral method gives a solution \citep{baaquie1997,pramanik2020,param2021}. Previous works using Feynman path integral method has been done in motor control theory by \cite{kappen2005}, \cite{theodorou2010} and \cite{theodorou2011}. Applications of Feynman path integral in finance has been discussed rigorously in \cite{belal2007}. In \cite{pramanik2020} a Feynman-type path integral has been introduced to determine a feedback stochastic control. This methods works in both linear and non-linear stochastic differential equations and a Fourier transformation has been used to find out solution of Wick-rotated Schr\"odinger type equation \citep{pramanik2020,pramanik2020motivation,pramanik2020optimization,param2021,pramanik2021opt,param2021s}. A more generalized Nash equilibrium on tensor field has been discussed in \cite{pramanik2019}.

\section{A stochastic differential game of opinion dynamics}

Following \cite{niazi2016} consider a social network of $n$ agents by a weighted directed graph $G=(N,E,w_{ij})$, where $N=\{1,...,n\}$ is the set of all agents. Suppose, $E\subseteq N\times N$ is the set of all ordered pairs of all connected agents and, $w_{ij}$ is the influence of agent $j$ on agent $i$ for all $(i,j)\in E$. There are usually two types of connections, one sided or two sided. For the principle-agent problem the connection is one sided (i.e. Stackelberg model) and agent-agent problem it is two sided (i.e. Cournot model). Suppose $x^i(s)\in[0,1]$ be the opinion of agent $i^{th}$ at time $s\in[0,t]$ with their initial opinion $x^i(0)=x_0^i\in[0,1]$. Then $x^i(s)$ has been normalized into $[0,1]$ where $x^i(s)=0$ stands for a strong disagreement and $x^i(s)=1$ represents strong agreement and all other agreements stays in between. Consider $\mathbf x(s)=\left[x^1(s),x^2(s),...,x^n(s)\right]'\in[0,1]^n$ be the opinion profile vector of $n$-agents at time $s$ where `prime' represents the transpose. Following \cite{niazi2016} consider a cost function of agent $i$ as
\begin{align}\label{0}
L^i(s,\mathbf x,x_0^i,u^i)&=\int_0^t \mbox{$\frac{1}{2}$} \bigg\{\sum_{j\in\eta_i}w_{ij}\left[x^i(s)-x^j(s)\right]^2+k_i\left[x^i(s)-x_0^i\right]^2+\left[u^i(s)\right]^2\bigg\} ds,
\end{align}
where $w_{ij}\in[0,\infty)$ is a parameter which weighs the susceptibility of agent $j$ to influence agent $i$, $k_i\in[0,\infty)$ is agent $i$'s stubbornness, $u^i(s)$ is the control variable of agent $i$ and set of all agents with whom $i$ interacts is $\eta_i$ and defined as $\eta_i:=\{j\in N:(i,j)\in E\}$. The cost function $L^i(s,\mathbf x,x_0^i,u^i)$ is twice differentiable with respect to time in order to satisfy Wick rotation, is continuously differentiable with respect to $i^{th}$ agent's control $u^i(s)$, non-decreasing in opinion $x^i(s)$, non-increasing in $u^i(s)$, and convex and continuous in all opinions and controls \citep{mas1995,pramanik2020optimization}. The opinion dynamics of agent $i$ follows a stochastic differential equation 
\begin{align}\label{1}
dx^i(s)&=\mu^i[s,x^i(s),u^i(s)]ds+\sigma^i[s,x^i(s),u^i(s)]dB^i(s),
\end{align}
with the initial condition $x_0^i$, where $\mu^i$ and $\sigma^i$ are the drift and diffusion component of agent $i$ with $B^i(s)$ is the Brownian motion. The reason behind incorporating Brownian motion in agent $i$'s opinion dynamics is because of Hebbian Learning which states that, neurons increase the synaptic connection strength between them when they are active together simultaneously and this behavior in probabilistic in the sense that, resource availability from a particular place is random  \citep{hebb2005, kappen2007}. For example, for a given stubbornness, and influence from agent $j$, agent $i$'s opinion dynamics has some randomness in opinion. Suppose, from other resources agent $i$ knows that, the information provided by agent $j$'s influence is misleading. Apart from that after considering humans as automatons, motor control and foraging for food becomes a big examples of minimization of costs (or the expected return) \cite{kappen2007}. As control problems like motor controls are stochastic in nature because there is a noise in the relation between the muscle contraction and the actual displacement with joints with the change of the information environment over time, we consider Feynman path integral approach to calculate the stochastic control after assuming the opinion dynamics Equation (\ref{1}) (\cite{feynman1949},\cite{fujiwara2017}). The coefficient of the control term in Equation (\ref{0}) is normalized to $1$, without loss of generality. The cost functional represented in the Equation (\ref{0}) is viewed as a model of the motive of agent $i$ towards a prevailing social issue \cite{niazi2016}. In this dynamic social network problem agent $i$'s objective is to $\min_{u^i}\{\E_s(L^i)|\mathcal F_0^x\}$ subject to the Equation (\ref{1}), where $\E_0(L^i)|\mathcal F_0^x$ represents the expectation on $L^i$ at time $0$ subject to agent $i$'s opinion filtration $\mathcal F_0^x$ starting at the initial time $0$. A solution of this problem is a feedback Nash equilibrium as the control of agent $i$ is updated based on the opinion at the same time $s$.

\section{Definitions and Assumptions}

\begin{as}\label{as0}
	For $t>0$ and $i=1,...,n$, let ${{\mu^i}}(s,x^i,u^i):[0,t]\times\mathbb{R}^2 \ra\mathbb{R}$ and ${\sigma}^i(s,x^i,u^i):[0,t]\times\mathbb{R}^2 \ra\mathbb{R}$ be some measurable function and, for some constant $M_1^i>0$ and, for opinion $x^i\in\mathbb{R}$ the linear growth of agent $i$'s control $u^i$ as
	\[
	|{{\mu^i}}(s,x^i,u^i)|+
	|{\sigma}^i(s,x^i,u^i)|\leq 
	M_1^i(1+|x^i|),
	\]
	such that, there exists another constant $M_2^i>0$ and for a different
	$\widetilde{x}^i\in\mathbb{R}$ such that the Lipschitz conditions,
	\[
	|{{\mu}^i}(s,x^i,u^i)-
	{{\mu}^i}(s,\widetilde{x}^i,u^i)|+|{\sigma}^i(s,x^i,u^i)-{\sigma}^i(s,u,\widetilde{x})|
	\leq M_2^i |x^i-\widetilde{x}^i|,
	\]
	and
	\[
	|{{\mu}^i}(s,x^i,u^i)|^2+
	|{\sigma}^i(s,x^i,u^i)|^2\leq (M_2^i)^2
	(1+|\widetilde{x}^i|^2),
	\]
	hold.
\end{as}

\begin{as}\label{asf1}
	Agent $i$ faces a probability space $(\Omega,\mathcal{F}_s^{x},\mathcal{P})$ with sample space $\Omega$, $u^i$-adaptive filtration at time $s$ of opinion $x^i$ as $\{\mathcal{F}_s^{x}\}\subset\mathcal{F}_s$, a probability measure $\mathcal{P}$ and $n$-dimensional $\{\mathcal{F}_s\}$ Brownian motion $B^i$ where the control of $i^{th}$ agent $u^i$ is an $\{\mathcal{F}_s^{x}\}$ adapted process such that Assumption \ref{as0} holds, for the feedback control measure of agents in a society there exists a measurable function $h^i$ such that $h^i:[0,t]\times C([0,t]):\mathbb{R}^{n}\ra u^i$ for which $u^i(s)=h^i[x^i(s,u^i)]$ such that Equation (\ref{1}) has a strong unique solution.
\end{as}

\begin{as}\label{asf3}
	(i). $\mathcal Z\subset\mathbb R^{n}$ such that agent $i$ cannot go beyond set $\mathcal Z_i\subset \mathcal Z$ because of their limitations of acquiring knowledge from their society at a given time. This immediately implies set $\mathcal Z_i$ is different for different agents. If the agent is young , they would have less limitation to acquire new information and make opinions on it.\\
	(ii). The function $h_0^i:[0,t]\times\mathbb R^{n}\ra\mathbb R$. Therefore, all agents in a society at the beginning of $[0,t]$ have the cost function $h_0:[0,t]\times\mathbb R^{n}\ra\mathbb R$ such that $h_0^i\subset h_0$ in functional spaces and both of them are concave which is equivalent to Slater condition \citep{marcet2019}. Possibility of giving a partial opinion has been omitted in this paper.\\
	(iii). There exists an $\epsilon>0$ with $\epsilon\downarrow 0$ for all $(x^i, u^i)$ and $i=1,2,...,n$ such that 
	\[
	\E_0\left\{\int_{0}^t\mbox{$\frac{1}{2}$} \bigg\{\sum_{j\in\eta_i}w_{ij}\left[x^i(s)-x^j(s)\right]^2+k_i\left[x^i(s)-x_0^i\right]^2+\left[u^i(s)\right]^2\bigg\}\bigg|\mathcal F_0^{x}\right\}ds\geq\epsilon.
	\]
\end{as}

The opinion dynamics of $i^{th}$ agent is continuous and it is mapped from an interval to a space of continuous functions with initial and terminal time points. Suppose, at time $s$, $g(s):[p,q]\ra\mathcal{C}$ represents an opinion dynamics of $i^{th}$ agent with initial and terminal points $g(p)$ and $g(q)$ respectively, such that, the line path integral is $\int_{\mathcal{C}} f(\gamma) ds=\int_{p}^q f(g(s))|g'(s)| ds$,  where $g'(s)$ is derivative with respect to $s$. In this paper I consider functional path integrals where the domain of the integral is the space of functions \citep{pramanik2020,param2021,pramanik2021opt}. Functional path integrals are very popular in probability theory and quantum mechanics. In \cite{feynman1948} theoretical physicist Richard Feynman introduced \emph{Feynman path integral} and popularized it in quantum mechanics. Furthermore, mathematicians develop the measurability of this functional integral and in recent years it has become popular in probability theory \citep{fujiwara2017}. In quantum mechanics, when a particle moves from one point to another, between those points it chooses the shortest path out of infinitely many paths such that some of them touch the edge of the universe. After introducing equal length small time interval$[s,s+\epsilon]$ with $\epsilon>0$ such that $\epsilon\downarrow 0$ and using Riemann–Lebesgue lemma if at time $s$ one particle touches the end of the universe, then at a later time point it would come back and go to the opposite side of the previous direction to make the path integral a measurable function \citep{bochner1949}. Similarly, agent $i$ has infinitely opinions, out of them they choose the opnion corresponding to least cost given by the constraint explained in Equation (\ref{1}). Furthermore, the advantage of Feynman approach is that, it can be used in both in linear and non-linear stochastic differential equation systems where constructing an HJB equation is almost impossible \citep{belal2007}.

\begin{definition}\label{d0}
	Suppose, for a particle $\hat{\mathcal{L}}[s,y(s),\dot y(s)]=(1/2) \hat m\dot y(s)^2-\hat{V}(y)$ be the Lagrangian in classical sense in generalized coordinate $y$ with mass $\hat m$ where $(1/2)\hat m\dot y^2$ and $\hat V(y)$ are kinetic and potential energies respectively. The transition function of Feynman path integral corresponding to the classical action function\\ $Z^*=\int_0^t \hat{\mathcal{L}}(s,y(s),\dot y(s)) ds$  is defined as $\Psi(y)=\int_{\mathbb R} \exp\{{Z^*}\} \mathcal{D}_Y $, where $\dot y=\partial y/\partial s$ and $\mathcal{D}_Y$ is an approximated Riemann measure which represents the positions of the particle at different time points $s$ in $[0,t]$ \citep{pramanik2020}.
\end{definition}

Here $i^{th}$ agent's objective is to minimize Equation (\ref{0}) subject to Equations (\ref{1}). Following Definition \ref{d0} the quantum Lagrangian at time $s$ of $[s,s+\epsilon]$ is 
\begin{multline}\label{3}
\mathcal{L}^i=\E_s\biggr\{\mbox{$\frac{1}{2}$} \bigg\{\sum_{j\in\eta_i}w_{ij}\left[x^i(s)-x^j(s)\right]^2+k_i\left[x^i(s)-x_0^i\right]^2+\left[u^i(s)\right]^2\bigg\}ds\\+\lambda^i\left[\Delta x^i(s)-\mu^i[s,x^i(s),u^i(s)]ds-\sigma^i[s,x^i(s),u^i(s)]dB^i(s)\right]\biggr\},
\end{multline}
where $\lambda^i$ is a time independent quantum Lagrangian multiplier (one can think of as a penalization constant of agent $i$). As at the beginning of the small time interval $[s,s+\epsilon]$, agent $i$ does not have any future information, they make expectations based on their opinion $x^i$. For another normalizing constant $L_\epsilon^i>0$ and for time interval $[s,s+\epsilon]$ such that $\epsilon\downarrow 0$ define a transition function from $s$ to $s+\epsilon$ as
\begin{equation}\label{lin8}
\Psi_{s,s+\varepsilon}^i(x^i)=\frac{1}{L_\varepsilon^i} \int_{\mathbb{R}^{n}}\exp[-\varepsilon \mathcal{A}_{s,s+\varepsilon}(x^i)]\Psi_s^i(x^i)dx^i(s),
\end{equation}
where $\Psi_s^i(x^i)$ is the value of the transition function based on opinion $x^i$ at time $s$ with the initial condition $\Psi_0^i(x^i)=\Psi_0^i$. Therefore, the action function of agent $i$ is,
\begin{multline*}
\mathcal{A}_{s,s+\varepsilon}(x^i)=
\int_{s}^{s+\varepsilon}\E_\nu\biggr\{\mbox{$\frac{1}{2}$} \bigg\{\sum_{j\in\eta_i}w_{ij}\left[x^i(\nu)-x^j(\nu)\right]^2+k_i\left[x^i(\nu)-x_0^i\right]^2+\left[u^i(\nu)\right]^2\bigg\}d\nu\\+h^i[\nu+\Delta \nu,x^i(\nu)+\Delta x^i(\nu)]\biggr\},
\end{multline*}
where $h^i[\nu+\Delta \nu,x^i(\nu)+\Delta x^i(\nu)]\in C^2([0,t]\times\mathbb{R}^{n})$ such that,
\begin{multline*}
h^i[\nu+\Delta \nu,x^i(\nu)+\Delta x^i(\nu)]\\ = \lambda^i\left[\Delta x^i(\nu)-\mu^i[\nu,x^i(\nu),u^i(\nu)]d\nu-\sigma^i[\nu,x^i(\nu),u^i(\nu)]dB^i(\nu)\right].
\end{multline*}
Here the action function has the notation $\mathcal{A}_{s,s+\epsilon}(x^i)$ which means within $[s,s+\epsilon]$ the action of agent $i$ depends on their opinion  $x^i$ and furthermore, I assume this system has a feedback structure. Therefore, the opinion of agent $i$ also depends on the strategy $u^i$ as well as the rest of the school. Same argument goes to the transition function $\Psi_{s,s+\epsilon}(x^i)$.

\begin{definition}\label{d1}
	For agent $i$ optimal opinion  $x^{i*}(s)$ and their continuous optimal strategy  $u^{i*}(s)$ constitute a dynamic stochastic Equilibrium  such that for all $s\in[0,t]$ the conditional expectation of the cost function is
	\begin{multline*}
	\E_0 \left[\int_{0}^t\mbox{$\frac{1}{2}$} \bigg\{\sum_{j\in\eta_i}w_{ij}\left[x^{i*}(s)-x^{j*}(s)\right]^2+k_i\left[x^{i*}(s)-x_0^i\right]^2+\left[u^{i*}(s)\right]^2\bigg\}ds\bigg|\mathcal F_0^{x^*} \right]ds\\\geq\E_0 \left[\int_{0}^t\mbox{$\frac{1}{2}$} \bigg\{\sum_{j\in\eta_i}w_{ij}\left[x^i(s)-x^j(s)\right]^2+k_i\left[x^i(s)-x_0^i\right]^2+\left[u^i(s)\right]^2\bigg\}ds\bigg|\mathcal F_0^{x}\right] ds,
	\end{multline*}
	with the opinion dynamics explained in Equation (\ref{1}), where $\mathcal F_0^{x^*}$ is the optimal filtration starting at time $0$ such that, $\mathcal F_0^{x^*}\subset\mathcal F_0^{x}$.
\end{definition}

\section{Main results}
Suppose, for the opinion space $S_0=\{\mathbf x(s):s\in[0,t]\}$ and agent $i$'s strategy space $\Gamma^i$ there exists a permissible  strategy $\gamma^i:[0,t]\times S_0\ra\Gamma^i$ and for all $i\in N$ define the integrand of the cost function as
\[
g^i(s,\mathbf x,x_0^i,u^i)=\mbox{$\frac{1}{2}$} \bigg(\sum_{j\in\eta_i}w_{ij}\left[x^i(s)-x^j(s)\right]^2+k_i\left[x^i(s)-x_0^i\right]^2+\left[u^i(s)\right]^2\bigg).
\]

\begin{prop}\label{p0}
	For stochastic dynamic game of $n$-agents of time interval $[0,t]$, let for agent $i$\\
	(i) the feedback control $u^i(s,x^i):[0,t]\times \mathbb{R}\ra\mathbb{R}$ is a continuously differentiable function,\\
	(ii) The cost integrand $g^i(s,\mathbf x,x_0^i,u^i):[0,t]\times\mathbb{R}^n\times\mathbb{R}\times\mathbb{R}\ra\mathbb{R}$ is a $C^2$ function on $\mathbb{R}$ for all $i\in N$.
	
	If $\big\{\gamma^{i*}(s,x_0^i,x^i(s))=\phi^{i*}(s,x^i); i\in N\big\}$ is a feedback Nash equilibrium and $\{\mathbf x(s), s\in[0,t]\}$ is the opinion trajectory, then there exists $n$ Lagrangian multipliers $\lambda^i:[0,t]\ra\mathbb{R}, i\in N$ with initial condition $\lambda_0^i$ such that, for a Lagrangian 
	\begin{align}
	\mathcal L^i(s,\mathbf x,x_0^i,u^i)&=g^i(s,\mathbf x,x_0^i,u^i)+\lambda^i\big[dx^i(s)-\mu^i(s,x^i,u^i)ds-\sigma^i(s,x^i,u^i)dB^i(s)\big]\notag
	\end{align}
	with its Euclidean action function 
	\begin{align}
	\mathcal A_{0,t}^i(x)&=\int_0^t\E_s\bigg\{g^i(s,\mathbf x,x_0^i,u^i)ds+\lambda^i\big[dx^i(s)-\mu^i(s,x^i,u^i)ds-\sigma^i(s,x^i,u^i)dB^i(s)\big]\bigg\}\notag
	\end{align}
	the following conditions hold: (a) $\lambda^i=\frac{\partial }{\partial x^i}\mathcal L^i$, and (b) $x^{i*}(0)=x_0^i\in[0,1]$ with $ i\in N$. Under this case, the optimal feedback control will be the solution of the following equation
	\begin{align}
	\mbox{$\frac{\partial}{\partial u^i}$}f^i(s,\mathbf x,\lambda^i,u^i) \left[\mbox{$\frac{\partial^2}{\partial (x^i)^2}$}f^i(s,\mathbf x,\lambda^i,u^i)\right]^2=2\mbox{$\frac{\partial}{\partial x^i}$}f^i(s,\mathbf x,\lambda^i,u^i) \mbox{$\frac{\partial^2}{\partial x^i\partial u^i}$}f^i(s,\mathbf x,\lambda^i,u^i),\notag
	\end{align}
	where for a function $h^i(s,x^i)\in C^2([0,\infty)\times\mathbb{R})$
	\begin{align}
	f^i(s,\mathbf x,\lambda^i,u^i)&=g^i(s,\mathbf x,x_0^i,u^i)+\lambda^ih^i(s,x^i)+\left[\lambda^i\mbox{$\frac{\partial h^i(s,x^i)}{\partial s}$}+\mbox{$\frac{\partial\lambda^i(s)}{\partial s}$}h^i(s,x^i)\right]\notag\\&\hspace{.25cm}+\lambda^i\mbox{$\frac{\partial h^i(s,x^i)}{\partial x^i}$}\mu^i(s,x^i,u^i)+\mbox{$\frac{1}{2}$}\lambda^i[\sigma^{i}(s,x^i,u^i)]^2\mbox{$\frac{\partial^2 h^i(s,x^i)}{\partial (x^i)^2}$}.\notag
	\end{align}
\end{prop}

\begin{proof}
	Equation (\ref{1}) implies
	\begin{align}\label{w14}
	x^i(s+ds)-x^i(s)&=\mu^i[s,x^i(s),u^i(s)]\ ds+\sigma^i[s,x^i(s),u^i(s)]\ dB^i(s).
	\end{align}
	Following \cite{chow1996} we get our Euclidean action function as 
	\begin{align}
	\mathcal A_{0,t}^i(x^i)&=\int_0^t\E_s\bigg\{g^i(s,\mathbf x,x_0^i,u^i)ds+\lambda^i(s)\big[dx^i(s)-\mu^i(s,x^i,u^i)ds-\sigma^i(s,x^i,u^i)dB^i(s)\big]\bigg\},\notag
	\end{align}
	where $E_s$ is the conditional expectation on opinion $x^i(s)$ at the beginning of time $s$. Now, for a small change in time $\Delta s=\varepsilon>0$, and for agent $i$'s normalizing constant $L_\varepsilon^i>0$ , define a transitional wave function in small time interval as
	\begin{align}\label{w16}
	\Psi_{s,s+\varepsilon}^i(x^i)&=\frac{1}{L_\varepsilon^i} \int_{\mathbb{R}} \exp\biggr\{-\varepsilon  \mathcal{A}_{s,s+\varepsilon}^i(x)\biggr\} \Psi_s^i(x^i) dx^i(s),
	\end{align}	
	for $\epsilon\downarrow 0$ and $\Psi_s^i(x^i)$ is the value of the transition function at time $s$ and opinion $x^i(s)$ with the initial condition $\Psi_0^i(x^i)=\Psi_0^i$ for all $i\in N$.
	
	For the small time interval $[s,\tau]$ where $\tau=s+\varepsilon$  the Lagrangian can be represented as,
	\begin{align}\label{action}
	\mathcal{A}_{s,\tau}^i(x)&= \int_{s}^{\tau}\ \E_s\ \biggr\{ g^i[\nu,\mathbf x(\nu),x_0^i,u^i(\nu)]\ d\nu+\lambda^i(\nu)\ \big[x^i(\nu+d\nu)-x^i(\nu)\notag\\&\hspace{2cm}-\mu^i[\nu,x^i(\nu),u^i(\nu)]\ d\nu-\sigma^i[\nu,x^i(\nu),u^i(\nu)]\ dB^i(\nu)\big] \biggr\},
	\end{align}
	with the initial condition $x^i(0)=x_0^i$.	This conditional expectation is valid when the control $u^i(\nu)$ of agent $i$'s opinion dynamics is determined at time $\nu$ and the opinions of all $n$-agents $\mathbf x(\nu)$ is given \citep{chow1996}. The evolution of a process takes place as if the action function is stationary. Therefore, the conditional expectation with respect to time only depends on the expectation of initial time point of interval $[s,\tau]$.
	
	Define $\Delta x^i(\nu)=x^i(\nu+d\nu)-x^i(\nu)$, then Fubini's Theorem implies,
	\begin{align}\label{action5}
	\mathcal{A}_{s,\tau}^i(x^i)&= \E_s\ \bigg\{ \int_{s}^{\tau}\ g^i[\nu,\mathbf x(\nu),x_0^i,u^i(\nu)]\ d\nu+\lambda^i(\nu)\ \big[\Delta x^i(\nu)\notag\\&\hspace{2cm}-\mu^i[\nu,x^i(\nu),u^i(\nu)]\ d\nu-\sigma^i[\nu,x^i(\nu),u^i(\nu)]\ dB^i(\nu)\big] \bigg\}.
	\end{align}
	By It\^o's Theorem there exists a function $h^i[\nu,x^i(\nu)]\in C^2([0,\infty)\times\mathbb{R})$ such that  $Y^i(\nu)=h^i[\nu,x^i(\nu)]$ where $Y^i(\nu)$ is an It\^o process \citep{oksendal2003}. After assuming \[
	h^i[\nu+\Delta \nu,x^i(\nu)+\Delta x^i(\nu)]= \Delta x^i(\nu)-\mu^i[\nu,x^i(\nu),u^i(\nu)]\ d\nu-\sigma^i[\nu,x^i(\nu),u^i(\nu)]\ dB^i(\nu),
	\]
	Equation (\ref{action5}) becomes,
	\begin{align}\label{action6}
	\mathcal{A}_{s,\tau}^i(x^i)&=\E_s \bigg\{ \int_{s}^{\tau}\ g^i[\nu,\mathbf x(\nu),x_0^i,u^i(\nu)]\ d\nu+\lambda^i h^i\left[\nu+\Delta \nu,x^i(\nu)+\Delta x^i(\nu)\right]\bigg\}.
	\end{align}
	For a very small interval around time point $s$ with $\varepsilon\downarrow 0$, and It\^o's Lemma yields,
	\begin{align}\label{action7}
	\varepsilon\mathcal{A}_{s,\tau}^i(x^i)&= \E_s \bigg\{\varepsilon g^i[s,\mathbf x(s),x_0^i,u^i(s)]+ \varepsilon\lambda^ih^i[s,x^i(s)]+ \varepsilon\lambda^ih_s^i[s,x^i(s)] \notag\\&\hspace{.25cm}+\varepsilon\lambda^ih_x^i[s,x^i(s)]\mu^i[s,x^i(s),u^i(s)] +\varepsilon\lambda^ih_x^i[s,x^i(s)]\sigma^i[s,x^i(s),u^i(s)] \Delta B^i(s)\notag\\&\hspace{.5cm}+\mbox{$\frac{1}{2}$}\varepsilon\lambda^i(\sigma^{i}[s,x^i(s),u^i(s)])^2h_{xx}^i[s,x^i(s)]+o(\varepsilon)\bigg\},
	\end{align}
	where $h_s^i=\frac{\partial}{\partial s} h^i$, $h_x^i=\frac{\partial}{\partial x^i} h^i$ and $h_{xx}^i=\frac{\partial^2}{\partial (x^i)^2} h^i$, and we use the condition $[\Delta x^i(s)]^2=\varepsilon$ with $\Delta x^i(s)=\varepsilon\mu^i[s,x^i(s),u^i(s)]+\sigma^i[s,x^i(s),u^i(s)]\Delta B^i(s)$. We use It\^o's Lemma and a similar approximation to approximate the integral. With $\varepsilon\downarrow 0$, dividing throughout by  $\varepsilon$ and taking the conditional expectation we get,
	\begin{align}\label{action8}
	\varepsilon\mathcal{A}_{s,\tau}^i(x^i)&= \E_s \bigg\{\varepsilon g^i[s,\mathbf x(s),x_0^i,u^i(s)]+\varepsilon\lambda^ih^i[s,x^i(s)]+ \varepsilon\lambda^ih_s^i[s,x^i(s)]\notag\\&\hspace{.25cm}+\varepsilon\lambda^ih_x^i[s,x^i(s)]\mu^i[s,x^i(s),u^i(s)]+\mbox{$\frac{1}{2}$}\varepsilon\lambda^i\sigma^{2i}[s,x^i(s),u^i(s)]h_{xx}^i[s,x^i(s)]+o(1)\bigg\},
	\end{align}
	as $\E_s[\Delta B^i(s)]=0$ and $\E_s[o(\varepsilon)]/\varepsilon\ra 0$ as $\varepsilon\downarrow 0$ with the initial condition $x_0^i$. For $\varepsilon\downarrow 0$ the transition function at $s$ is $\Psi_s^i(x^i)$ for all $i\in N$. Hence, using Equation (\ref{w16}), the transition function for $[s,\tau]$ is
	\begin{multline}\label{action9}
	\Psi_{s,\tau}^i(x^i)=\frac{1}{L_\epsilon^i}\int_{\mathbb{R}} \exp\biggr\{-\varepsilon \big[g^i[s,\mathbf x(s),x_0^i,u^i(s)]+\lambda^ih^i[s,x^i(s)]\\ +\lambda^ih_s^i[s,x^i(s)] +\lambda^ih_x^i[s,x^i(s)]\mu^i[s,x^i(s),u^i(s)]\\+\mbox{$\frac{1}{2}$}\lambda^i(s)(\sigma^{i}[s,x^i(s),u^i(s)])^2h_{xx}^i[s,x^i(s)]\big]\biggr\} \Psi_s^i(x) dx^i(s)+o(\varepsilon^{1/2}).
	\end{multline}
	As $\varepsilon\downarrow 0$, first order Taylor series expansion on the left hand side of Equation (\ref{action9}) gives
	\begin{multline}\label{action10}
	\Psi_{is}(x^i)+\varepsilon  \frac{\partial \Psi_{is}(x^i) }{\partial s}+o(\varepsilon)\\=\frac{1}{L_\varepsilon^i}\int_{\mathbb{R}} \exp\biggr\{-\varepsilon \big[g^i[s,\mathbf x(s),x_0^i,u^i(s)]+\lambda^i(s)h^i[s,x^i(s)] \\+\lambda^ih_s^i[s,x^i(s)]+\lambda^ih_x^i[s,x^i(s)]\mu^i[s,x^i(s),u^i(s)]\\+\mbox{$\frac{1}{2}$}\lambda^i(\sigma^{i}[s,x^i(s),u^i(s)])^2h_{xx}^i[s,x^i(s)]\big]\biggr\} \Psi_s^i(x) dx^i(s)+o(\varepsilon^{1/2}).
	\end{multline}
	For fixed $s$ and $\tau$ let $x^i(s)-x^i(\tau)=\xi^i$ so that $x^i(s)=x^i(\tau)+\xi^i$. When $\xi^i$ is not around zero, for a positive number $\eta<\infty$ we assume $|\xi^i|\leq\sqrt{\frac{\eta\varepsilon}{x^i(s)}}$ so that for $\varepsilon\downarrow 0$, $\xi^i$ takes even smaller values and agent $i$'s opinion $0<x^i(s)\leq\eta\varepsilon/(\xi^i)^2$. Therefore,
	\begin{multline*}
	\Psi_{is}(x^i)+\varepsilon\frac{\partial \Psi_{is}(x^i)}{\partial s}=\frac{1}{L_\epsilon^i}\int_{\mathbb{R}} \left[\Psi_{is}(x^i)+\xi^i\frac{\partial \Psi_{is}(x^i)}{\partial x^i}+o(\epsilon)\right]\\\times \exp\biggr\{-\varepsilon \big[g^i[s,\mathbf x(s),x_0^i,u^i(s)]+\lambda^ih^i[s,x^i(s)]+\lambda^ih_x^i[s,x^i(s)]\mu^i[s,x^i(s),u^i(s)]\\+\mbox{$\frac{1}{2}$}\lambda^i(\sigma^{i}[s,x^i(s),u^i(s)])^2h_{xx}^i[s,x^i(s)]\big]\biggr\} d\xi^i+o(\varepsilon^{1/2}).
	\end{multline*}
	Before solving for Gaussian integral of the each term of the right hand side of the above Equation define a $C^2$ function 
	\begin{align*}
	& f^i[s,\bm\xi,\lambda^i(s),u^i(s)]\notag\\&=g^i[s,\mathbf x(s)+\bm{\xi},x_0^i,u^i(s)]+\lambda^ih^i[s,x^i(s)+\xi^i] +\lambda^ih_s^i[s,x^i(s)+\xi^i]\notag\\&\hspace{.25cm}+\lambda^ih_x^i[s,x^i(s)+\xi^i]\mu^i[s,x^i(s)+\xi^i,u^i(s)]+\mbox{$\frac{1}{2}$}\lambda^i\sigma^{2i}[s,x^i(s)+\xi^i,u^i(s)]h_{xx}^i[s,x^i(s)+\xi^i]+o(1),
	\end{align*}
	where $\bm{\xi}$ is a vector of all $n$-agents' $\xi^i$'s.  Hence,
	\begin{align}\label{action13}
	\Psi_{is}(x^i)+\varepsilon \frac{\partial \Psi_{is}(x^i) }{\partial s}&=\Psi_{is}(x^i)\frac{1}{L_\epsilon^i}\int_{\mathbb{R}}\exp\left\{-\varepsilon f^i[s,\bm\xi,\lambda^i(s),u^i(s)] \right\}d\xi^i\notag\\&+\frac{\partial \Psi_{is}(x^i)}{\partial x^i}\frac{1}{L_\epsilon^i}\int_{\mathbb{R}}\xi^i\exp\left\{-\varepsilon f^i[s,\bm\xi,\lambda^i(s),u^i(s)] \right\}d\xi^i+o(\varepsilon^{1/2}).
	\end{align}
	After taking $\varepsilon\downarrow 0$, $\Delta u\downarrow0$ and a Taylor series expansion with respect to $x^i$ of $f^i[s,\bm\xi,\lambda^i,u^i(s)]$ gives, 
	\begin{align*}
	f^i[s,\bm\xi,\lambda^i,u(s)]&=f^i[s,\mathbf x(\tau),\lambda^i,u^i(s)]+f_x^i[s,\mathbf x(\tau),\lambda^i,u^i(s)][\xi^i-x^i(\tau)]\notag\\&\hspace{1cm}+\mbox{$\frac{1}{2}$}f_{xx}^i[s,\mathbf x(\tau),\lambda^i,u^i(s)][\xi^i-x^i(\tau)]^2+o(\varepsilon).
	\end{align*}
	Define $m^i=\xi^i-x^i(\tau)$ so that $ d\xi^i=dm^i$. First integral on the right hand side of Equation (\ref{action13}) becomes,
	\begin{align}\label{action14}
	&\int_{\mathbb{R}} \exp\big\{-\varepsilon f^i[s,\bm\xi,\lambda^i,u^i(s)]\} d\xi^i\notag\\&=\exp\big\{-\varepsilon f^i[s,\mathbf x(\tau),\lambda^i,u^i(s)]\big\}\notag\\&\hspace{1cm}\int_{\mathbb{R}} \exp\biggr\{-\varepsilon \biggr[f_x^i[s,\mathbf x(\tau),\lambda^i,u^i(s)]m^i+\mbox{$\frac{1}{2}$}f_{xx}^i[s,\mathbf x(\tau),\lambda^i,u^i(s)](m^i)^2\biggr]\biggr\} dm^i.
	\end{align}
	Assuming  $a^i=\frac{1}{2} f_{xx}^i[s,\mathbf x(\tau),\lambda^i,u^i(s)]$ and $b^i=f_x^i[s,\mathbf x(\tau),\lambda^i,u^i(s)]$ the argument of the exponential function in Equation (\ref{action14}) becomes,
	\begin{align}\label{action15}
	a^i(m^i)^2+b^im^i&=a^i\left[(m^i)^2+\frac{b^i}{a^i}m^i\right]=a^i\left(m^i+\frac{b^i}{2a^i}m^i\right)^2-\frac{(b^i)^2}{4(a^i)^2}.
	\end{align}
	Therefore,
	\begin{align}\label{action16}
	&\exp\big\{-\varepsilon f^i[s,\mathbf x(\tau),\lambda^i,u^i(s)]\big\}\int_{\mathbb{R}} \exp\big\{-\varepsilon [a^i(m^i)^2+b^im^i]\big\}dm^i\notag\\&=\exp\left\{\varepsilon \left[\frac{(b^i)^2}{4(a^i)^2}-f^i[s,\mathbf x(\tau),\lambda^i,u^i(s)]\right]\right\}\int_{\mathbb{R}} \exp\left\{-\left[\varepsilon a^i\left(m^i+\frac{b^i}{2a^i}m^i\right)^2\right]\right\} dm^i\notag\\&=\sqrt{\frac{\pi}{\varepsilon a^i}}\exp\left\{\varepsilon \left[\frac{(b^i)^2}{4(a^i)^2}-f^i[s,\mathbf x(\tau),\lambda^i,u^i(s)]\right]\right\},
	\end{align}
	and
	\begin{align}\label{action17}
	&\Psi_{is}(x^i) \frac{1}{L_\varepsilon^i} \int_{\mathbb{R}} \exp\big\{-\varepsilon f^i[s,\bm\xi,\lambda^i,u^i(s)]\} d\xi^i\notag\\ &=\Psi_{is}(x) \frac{1}{L_\varepsilon^i} \sqrt{\frac{\pi}{\varepsilon a^i}}\exp\left\{\varepsilon \left[\frac{(b^i)^2}{4(a^i)^2}-f^i[s,\mathbf x(\tau),\lambda^i,u^i(s)]\right]\right\}. 
	\end{align}
	Substituting $\xi^i=x^i(\tau)+m^i$ second integrand of the right hand side of Equation (\ref{action13}) yields,
	\begin{align}\label{action18}
	& \int_{\mathbb{R}} \xi^i \exp\left[-\varepsilon \{f^i[s,\bm\xi,\lambda^i,u^i(s)]\}\right] d\xi^i\notag\\&=\exp\{-\varepsilon f^i[s,\mathbf x(\tau),\lambda^i,u^i(s)]\}\int_{\mathbb{R}} [x^i(\tau)+m^i] \exp\left[-\varepsilon \left[a^i(m^i)^2+b^im^i\right]\right] dm^i\notag\\&=\exp\left\{\varepsilon \left[\frac{(b^i)^2}{4(a^i)^2}-f^i[s,\mathbf x(\tau),\lambda^i,u^i(s)]\right]\right\} \biggr[x^i(\tau)\sqrt{\frac{\pi}{\varepsilon a^i}}\notag\\&\hspace{1cm}+\int_{\mathbb{R}} m^i \exp\left\{-\varepsilon \left[a^i\left(m^i+\frac{b^i}{2a^i}m^i\right)^2\right]\right\} dm^i\biggr].
	\end{align}
	Substituting $k^i=m^i+b^i/(2a^i)$ in Equation (\ref{action18}) we get,
	\begin{align}\label{action19}
	&\exp\left\{\varepsilon \left[\frac{(b^i)^2}{4(a^i)^2}-f^i[s,\mathbf x(\tau),\lambda^i,u^i(s)]\right]\right\} \biggr[x^i(\tau)\sqrt{\frac{\pi}{\varepsilon a^i}}+\int_{\mathbb{R}} \left(k^i-\frac{b^i}{2a^i}\right) \exp[-a^i\varepsilon (k^i)^2] dk^i\biggr]\notag\\&=\exp\left\{\varepsilon \left[\frac{(b^i)^2}{4(a^i)^2}-f^i[s,\mathbf x(\tau),\lambda^i,u^i(s)]\right]\right\} \biggr[x^i(\tau)-\frac{b^i}{2a^i}\biggr]\sqrt{\frac{\pi}{\varepsilon a^i}}.
	\end{align}
	Hence,
	\begin{align}\label{action20}
	&\frac{1}{L_\varepsilon^i}\frac{\partial \Psi_{is}(x^i)}{\partial x^i}\int_{\mathbb{R}} \xi^i \exp\left[-\varepsilon f[s,\bm\xi,\lambda^i,u^i(s)]\right] d\xi^i\notag\\&=\frac{1}{L_\varepsilon^i}\frac{\partial \Psi_{is}(x^i)}{\partial x^i} \exp\left\{\varepsilon \left[\frac{(b^i)^2}{4(a^i)^2}-f^i[s,\mathbf x(\tau),\lambda^i,u^i(s)]\right]\right\} \biggr[x^i(\tau)-\frac{b^i}{2a^i}\biggr]\sqrt{\frac{\pi}{\varepsilon a^i}}.
	\end{align}
	Using results of Equations (\ref{action17}), and (\ref{action20})  into Equation (\ref{action13}) we get,
	\begin{align}\label{action24}
	& \Psi_{is}(x^i)+\varepsilon \frac{\partial \Psi_{is}(x^i)}{\partial s}\notag\\&=\frac{1}{L_\varepsilon^i} \sqrt{\frac{\pi}{\varepsilon a^i}}\Psi_{is}(x^i) \exp\left\{\varepsilon \left[\frac{(b^i)^2}{4(a^i)^2}-f^i[s,\mathbf x(\tau),\lambda^i,u^i(s)]\right]\right\}\notag\\&+\frac{1}{L_\varepsilon^i}\frac{\partial \Psi_{is}(x^i)}{\partial x^i} \sqrt{\frac{\pi}{\varepsilon a^i}} \exp\left\{\varepsilon \left[\frac{(b^i)^2}{4(a^i)^2}-f^i[s,\mathbf x(\tau),\lambda^i,u^i(s)]\right]\right\} \biggr[x^i(\tau)-\frac{b^i}{2a^i}\biggr]+o(\varepsilon^{1/2}).
	\end{align}
	As $f^i$ is in Schwartz space, derivatives are rapidly falling  and  assuming $0<|b^i|\leq\eta\varepsilon$, $0<|a^i|\leq\mbox{$\frac{1}{2}$}[1-(\xi^i)^{-2}]^{-1}$ and $x^i(s)-x^i(\tau)=\xi^i$ we get,
	\begin{align*}
	x^i(\tau)-\frac{b^i}{2a^i}=x^i(s)-\xi^i-\frac{b^i}{2a^i}=x^i(s)-\frac{b^i}{2a^i},
	\end{align*}
	such that 
	\begin{align*}
	\bigg|x^i(s)-\frac{b^i}{2a^i}\bigg|=\bigg|\frac{\eta\varepsilon}{(\xi^i)^2}-\eta\varepsilon\left[1-\frac{1}{(\xi^i)^2}\right]\bigg|\leq\eta\varepsilon.
	\end{align*}
	Therefore, Wick-rotated Schr\"odinger type Equation for agent $i$ is,
	\begin{align}\label{action25.4}
	\frac{\partial \Psi_{is}(x)}{\partial s}&=\left[\frac{(b^i)^2}{4(a^i)^2}-f^i[s,\mathbf x(\tau),\lambda^i(s),u^i(s)]\right]\Psi_{is}(x).
	\end{align}
	Differentiating the Equation (\ref{action25.4}) with respect to $u^i$ gives us optimal control of agent $i$ under this stochastic opinion dynamics which is
	\begin{align}\label{w18}
	\left\{\frac{2f_x^i}{f_{xx}^i}\left[\frac{f_{xx}^if_{xu}^i-f_x^if_{xxu}^i}{(f_{xx}^i)^2}\right]-f_u^i\right\}\Psi_{is}(x)=0,
	\end{align}
	where $f_x^i=\frac{\partial}{\partial x^i} f^i$, $f_{xx}^i=\frac{\partial^2}{\partial (x^i)^2} f^i$, $f_{xu}^i=\frac{\partial^2}{\partial x^i\partial u^i} f^i$ and $f_{xxu}^i=\frac{\partial^3}{\partial (x^i)^2\partial u^i} f^i=0$. Therefore, optimal feedback control of agent $i$ in stochastic opinion dynamics is represented as $\phi^{i*}(s,x^i)$ and is found by setting Equation (\ref{w18}) equal to zero. Hence, $\phi^{i*}(s,x^i)$ is the solution of the following Equation
	\begin{align}\label{w21}
	f_u^i (f_{xx}^i)^2=2f_x^i f_{xu}^i.
	\end{align}
\end{proof}

\begin{prop}\label{p1}
	For the initial condition $\Psi_0^i(x^i)=I^i(x^i)$ the Wick-rotated Schr\"odinger-type equation of agent $i\in N$ 
	\begin{align}
	\frac{\partial\Psi_{is}(x^i)}{\partial s}&=\left[\frac{(b^i)^2}{4(a^i)^2}-f^i[s,\mathbf x(s),\lambda^i,u^i(s)]\right]\Psi_{is}(x^i),\notag
	\end{align}
	where $a^i=\frac{1}{2} \frac{\partial^2}{\partial (x^i)^2}f^i[s,\mathbf x(s),\lambda^i,u^i(s)]$ and $b^i=\frac{\partial}{\partial x^i}f^i[s,\mathbf x(s),\lambda^i,u^i(s)]$, has a unique solution 
	\begin{align}
	\Psi_{is}(x)&=I^i(x^i)\exp\left\{s\left[\frac{(b^i)^2}{4(a^i)^2}-f^i[s,\mathbf x(s),\lambda^i,u^i(s)]\right]\right\}.\notag
	\end{align}
	The optimal opinion $x^{i*}$ can be found after solving the following equation,
	\begin{align}\label{op0}
	\frac{\partial^2}{\partial s\partial x^i}f^i[s,\mathbf x(s),\lambda^i,u^i(s)]&=\frac{\partial}{\partial x^i}f^i[s,\mathbf x(s),\lambda^i,u^i(s)], 
	\end{align}
	and corresponding feedback control Nash equilibrium is $\phi^{i*}(s,x^{i*})$.
\end{prop}

\begin{proof}
	Let for three variables $v^i[x^i(s),u^i(s)], z^i[x^i(s),u^i(s)]$ and $w^i[x^i(s),u^i(s)]$ generalized Wick-rotated Schr\"odinger type equation for agent $i$ is,
	\begin{align}\label{action27}
	\frac{\partial \Psi_{is}(x^i)}{\partial s}=v^i[x^i(s),u^i(s)]\Psi_{is}(x^i)+z^i[x^i(s),u^i(s)]\frac{\partial \Psi_{is}(x^i)}{\partial x^i}+w^i[x^i(s),u^i(s)]\frac{\partial^2 \Psi_{is}(x^i)}{\partial (x^i)^2},
	\end{align}
	with the initial condition $\Psi_0^i(x^i)=I^i(x^i)$. As agent $i$'s wave function $\Psi_{is}(x^i)$ is a function of opinion $x^i(s)$ for fixed control $u^i(s)$, the solution to Equation (\ref{action27}) is found by assuming $v^i, z^i$ and $w^i$ vary according to the movement of $x^i$'s only. Define $\Psi_{is;s}(x^i)=\frac{\partial}{\partial s}\Psi_{is}(x^i)$, $\Psi_{is;x}(x^i)=\frac{\partial}{\partial x^i}\Psi_{is}(x^i)$ and $\Psi_{is;xx}(x^i)=\frac{\partial^2}{\partial (x^i)^2}\Psi_{is}(x^i)$. Hence,
	\begin{align}\label{action280}
	\Psi_{is;s}(x^i)&=v^i(x^i,u^i)\Psi_{is}(x^i)+z^i(x^i,u^i)\Psi_{is;x}(x^i)+w^i(x^i,u^i)\Psi_{is;xx}(x^i).
	\end{align}
	For a $\tilde{\xi}\in\mathbb{R}$, the Fourier transformation of $\Psi_{is}(x^i)$ is,
	\begin{align}\label{action29}
	{B}\{\Psi_{is}(x^i)\}=\overline{\Psi}_s(\tilde{\xi})=\int_{\mathbb{R}} \Psi_{is}(x^i) \exp\big\{-\mathfrak{i}\tilde{\xi}x^i\big\} dx^i.
	\end{align} 
	As ${B}\{\Psi_{is;x}(x^i)\}=\int_{\mathbb{R}} \frac{\partial }{\partial x^i}\Psi_{is}(x^i) \exp\{-\mathfrak{i}\tilde{\xi}x^i\} dx^i$ then assuming  $\Psi_{is}(x^i)\downarrow 0$ as $x^i\ra\pm\infty$, Equation (\ref{action29}) gives, ${B}\{\Psi_{is;x}(x^i)\}=\mathfrak{i}\tilde{\xi} {B}\{\Psi_{is}(x^i)\}$. Therefore, ${B}\{\Psi_{is;x}(x^i)\}=\mathfrak{i}\tilde{\xi} {B}\{\Psi_{is}(x^i)\}$ and, ${B}\{\Psi_{is;xx}(x^i)\}=\mathfrak{i}\tilde{\xi}{B}\{\Psi_{is;x}(x^i)\}=-\tilde{\xi}^2 {B}\{\Psi_{is}(x^i)\}$. Rearranging terms in Equation (\ref{action280}) and  Fourier transformation with above conditions give,
	\begin{align}\label{action31}
	&\Psi_{is;s}(x^i)-v^i(x^i,u^i)\Psi_{is}(x^i)-z^i(x^i,u^i)\Psi_{is;x}(x^i)-w^i(x^i,u^i)\Psi_{is;xx}(x^i)=0\notag\\& \frac{\partial \overline{\Psi}_{is}(\tilde{\xi})}{\partial s}+\overline{\Psi}_{is}(\tilde{\xi})\big[w^i(x^i,u^i)\tilde{\xi}^2-z^i(x^i,u^i)\mathfrak{i}\tilde{\xi}-v^i(x^i,u^i)\big]=0.
	\end{align}
	Let us assume an integrating factor $\exp\left\{\int [w^i(x^i,u^i)\tilde{\xi}^2-z^i(x^i,u^i)\mathfrak{i}\tilde{\xi}-v^i(x^i,u^i)] ds\right\}$ which can be written as $\exp\left\{s[w^i(x^i,u^i)\tilde{\xi}^2-z^i(x^i,u^i)\ \mathfrak{i}\tilde{\xi}-v^i(x^i,u^i)]\right\}.$ Therefore,
	\begin{align}
	& \exp\left\{s\left[w^i\tilde{\xi}^2- z \mathfrak{i}\tilde{\xi}-v^i\right]\right\}\left\{\mbox{$\frac{\partial}{\partial s}$}\overline{\Psi}_{is}(\tilde{\xi})+\overline{\Psi}_{is}(\tilde{\xi})\big[w^i\tilde{\xi}^2-z^i\mathfrak{i}\tilde{\xi}-v^i\big]\right\}=0,\notag
	\end{align}
	or equivalently
	\begin{multline*}
	\mbox{$\frac{\partial}{\partial s}$}\overline{\Psi}_{is}(\tilde{\xi}) \exp\left\{s\left[w^i\tilde{\xi}^2- z^i \mathfrak{i}\tilde{\xi}-v^i\right]\right\}+\\\left\{\overline{\Psi}_{is}(\tilde{\xi})\big[w^i\tilde{\xi}^2-z^i\mathfrak{i}\tilde{\xi}-v^i\big]\right\} \exp\left\{s\left[w^i\tilde{\xi}^2-z^i \mathfrak{i}\tilde{\xi}-v^i\right]\right\}=0,
	\end{multline*}	
	so that
	\begin{align}\label{action32}
	\mbox{$\frac{\partial}{\partial s}$} \exp\left\{s\left[w^i\tilde{\xi}^2-z^i \mathfrak{i}\tilde{\xi}-v^i\right]\right\} \overline{\Psi}_{is}(\tilde{\xi})=0.
	\end{align}
	Integrating both sides of Equation (\ref{action32}) yields,
	\begin{align}\label{action33}
	\exp\left\{s\left[w^i\tilde{\xi}^2-z^i \mathfrak{i}\tilde{\xi}-v^i\right]\right\} \overline{\Psi}_{is}(\tilde{\xi})=c^i(\tilde{\xi})\ \text{and,}\notag\\ \overline{\Psi}_{is}(\tilde{\xi})=c^i(\tilde{\xi})\ \exp\left\{-s\left[w^i\tilde{\xi}^2- z^i \mathfrak{i}\tilde{\xi}-v^i\right]\right\}.
	\end{align}
	Applying the Fourier transformation on the initial condition yields,
	$\overline{\Psi}_0^i(\tilde{\xi})=\overline{I}^i(\tilde{\xi})$ which implies $c^i(\tilde{\xi})\equiv \overline{I}^i{(\tilde{\xi})}$.
	Using this condition  Equation (\ref{action33}) gives,
	\begin{align}\label{action35}
	\overline{\Psi}_{is}(\tilde{\xi})&=\overline{I}^i(\tilde{\xi}) \exp\left\{-s\left[w^i\tilde{\xi}^2-z^i \mathfrak{i}\tilde{\xi}-v^i\right]\right\}=\overline{I}^i(\tilde{\xi}) \overline {\Phi}^i(s,\tilde{\xi}),
	\end{align}
	where $\overline {\Phi}^i(s,\tilde{\xi})=\exp\left\{-s\left[w^i\tilde{\xi}^2- z^i \mathfrak{i}\tilde{\xi}-v^i\right]\right\}$ for all $i\in N$. Fourier Inversion Theorem yields,
	\begin{align}\label{action36}
	\Phi^i(s,x^i)&=\frac{1}{2\pi} \exp\left[\frac{(sz^i-x^i)^2}{4s^2(w^i)^2}+sv^i\right]\sqrt{\frac{\pi}{sw^i}},\ \forall i\in N.
	\end{align}
	As the Fourier transformation $\overline{\Psi}_{is}(\tilde{\xi})=\overline{I}^i(\tilde{\xi}) \overline{\Phi}^i(s,\tilde{\xi})$ is the product of two Fourier transformations, therefore the Convolution Theorem implies that for $I^i(x^i)$ and
	$\Phi^i[s,x^i(s)]$, 
	\[
	\overline{\Psi}_{is}(\tilde{\xi})={B}\left\{I^i[x^i(s)]*\Phi^i[s,x^i(s)]\right\},
	\]
	and 
	\[
	\Psi_{is}(x^i)=(I^i*\Phi^i)[s,x^i(s)]=\int_{\mathbb{R}} \Phi^i(s,x^i-y^i) I^i(y^i) dy^i,
	\] for all $y^i\in\mathbb{R}$. Hence, a solution to the Equation (\ref{action27}) is,
	\begin{multline}\label{action40}
	\Psi_{is}(x^i)=\int_{\mathbb{R}} \frac{1}{2\pi}\ \exp\left\{\frac{[sz^i(x^i-y^i,u^i)-(x^i-y^i)]^2}{4s^2(w^i)^2(x^i-y^i,u^i)}+sv^i(x^i-y^i,u^i)\right\}\\\times\sqrt{\frac{\pi}{sw^i(x^i-y^i,u^i)}} I^i(y^i) dy^i.
	\end{multline}
	If one compares Wick-rotated Schr\"odinger type Equation (\ref{action25.4}) with (\ref{action27}) we find out $v^i(x^i,u^i)=(b^i)^2/[4(a^i)^2]-f^i(s,x^i,u^i)$ and other terms vanishes. Therefore, Equation (\ref{action36}) becomes
	\begin{align}\label{action36.1}
	\Phi^i(s,x^i)&=\frac{1}{2\pi} \int_{\mathbb{R}} \exp(sv^i) \exp\left(\mathfrak{i}\tilde{\xi}x^i\right) d\tilde{\xi} =\exp(sv^i) \delta(x^i),
	\end{align}
	where $\delta(x^i)=\frac{1}{2\pi} \int_{\mathbb{R}} \exp\left(\mathfrak{i}\tilde{\xi}x^i\right) d\tilde{\xi}$ is the Dirac $\delta$-function of the opinion of agent $i$. Now,
	\begin{multline*}
	\overline{\Psi}_{is}(\tilde{\xi})=\int_{\mathbb{R}} \int_{\mathbb{R}} \exp(-\mathfrak{i}\tilde{\xi}x^i) I^i(y^i)*\Phi(x^i-y^i)dx^i dy^i\\ =\int_{\mathbb{R}} I^i(y^i) \left[\int_{\mathbb{R}} \exp(-\mathfrak{i}\tilde{\xi}x^i) \Phi(x^i-y^i) dx^i\right] dy^i.
	\end{multline*}
	Suppose, $u^i=x^i-y^i$ such that $du^i=dx^i$ for all $i\in N$. Then
	\begin{align}
	\overline{\Psi}_{is}(\tilde{\xi})&=\int_{\mathbb{R}} I^i(y^i) \left[\int_{\mathbb{R}} \exp(-\mathfrak{i}\tilde{\xi}u^i) \exp(-\mathfrak{i}\tilde{\xi}y^i) \Phi(u^i)du^i \right] dy^i\notag\\&=\left[\int_{\mathbb{R}} I^i(y^i) \exp(-\mathfrak{i}\tilde{\xi}y^i) dy^i\right] \left[\int_{\mathbb{R}} \Phi(u^i) \exp(-\mathfrak{i}\tilde{\xi}u^i) du^i \right]=B(I^i)*B(\Phi^i).\notag 
	\end{align}
	Therefore, the solution to Equation (\ref{action25.4}) is $\Psi_{is}(x^i)=I^i(x^i) \exp[sv^i(x^i,u^i)]$ where $v^i(x^i,u^i)=(b^i)^2/[4(a^i)^2]-f^i(s,x^i,u^i)$. After using this solution to the wave function $\Psi_{is}(x^i)$ into Wick-rotated Schr\"odinger type Equation (\ref{action25.4}) we get,
	\begin{align}
	\mbox{$\frac{\partial}{\partial s}$}f^i[s,\mathbf x(s),\lambda^i,u^i(s)]&=\mbox{$\frac{\partial}{\partial x^i}$}f^i[s,\mathbf x(s),\lambda^i,u^i(s)], \notag
	\end{align}
	and differentiating with respect to $x^i$ gives
	\begin{align}\label{action36.10}
	\mbox{$\frac{\partial}{\partial x^i}$}\left\{\mbox{$\frac{\partial}{\partial s}$}f^i[s,\mathbf x(s),\lambda^i,u^i(s)]\right\}&=\mbox{$\frac{\partial}{\partial x^i}$}f^i[s,\mathbf x(s),\lambda^i,u^i(s)].
	\end{align}
	Optimal opinion of agent $i$, $x^{i*}$ can be found after solving the Equation (\ref{action36.10}) and an optimal feedback control $\phi^{i*}(s,x^{i*})$ is obtained.
\end{proof}

\begin{cor}\label{c0}
	Define $\mathbf x^*=[x^{1*},x^{2*},...,x^{n*}]^T$ for all $i\in N$. As each player has an optimal opinion $x^{i*}$, $\mathbf x^*$ is an optimal opinion vector. Furthermore, $$\phi^*(s,\mathbf x^*)=\left[\phi^{1*}(s,x^{1*}),\phi^{2*}(s,x^{2*}),...,\phi^{n*}(s,x^{n*})\right]^T$$ is an optimal control vector of feedback Nash equilibrium.
\end{cor}

After combining the opinion state variables and the Lagrangian multipliers, the following equation is obtained
\[
\begin{bmatrix}
d\mathbf x(s)\\d\bm{\lambda(s)}
\end{bmatrix}=\hat{\mathbf K}
\begin{bmatrix}
\mathbf x_0\\\bm{\lambda}_0
\end{bmatrix}ds+\mathbf A
\begin{bmatrix}
\mathbf x(s)\\\bm{\lambda}(s)
\end{bmatrix}ds+
\begin{bmatrix}
\bm{\sigma}\\\bm 0
\end{bmatrix}
\begin{bmatrix}
d\mathbf B(s)\\ d\mathbf B_{\bm{\lambda}}(s)
\end{bmatrix}
\]
where
\[
\mathbf A=\begin{bmatrix}
\bm{\mu}\ \ -\mathbf I\\-\mathbf W\ \ \bm 0
\end{bmatrix}, \ \hat{\mathbf K}=
\begin{bmatrix}
\bm 0\ \ \bm 0\\\mathbf K\ \ \bm 0
\end{bmatrix}
\]
where $\mathbf I$ is the identity matrix of size $n$, $\bm{\lambda}(s)=[\lambda^1(s),\lambda^2(s),...,\lambda^n(s)]^T$, \\$\mathbf K=\text{diag}[k_1,k_2,...,k_n]$, $\bm{\mu}$ is an $n\times 1$ vector, $\bm{\sigma}$ is an $n\times m$-dimensional matrix $d\mathbf B(s)$ is an $m\times 1$-dimensional Brownian motion corresponding to opinion and $d\mathbf B_{\bm{\lambda}}(s)$ is an $m\times 1$ dimensional Brownian motion of the Lagrangian multiplier. Following \cite{niazi2016}
\[
\mathbf W=
\begin{bmatrix}
q_1& -w_{12} & \dots &-w_{1n}\\
-w_{21} & q_2 & \dots & -w_{2n}\\
\vdots &\vdots&\ddots &\vdots\\
-w_{n1} & -w_{n2} & \dots & q_n
\end{bmatrix}
\]
with $q_i=\sum_{j\in\eta_i}w_{ij}+k_i$. $\mathbf W$ is a Laplacian-like matrix of a weighted directed gaph $G$ \citep{niazi2016} where $ij^{th}$ element in the off-diagonal shows the weight of the edge directed from $i$ to $j$. Define $d\mathbf X(s)=[d\mathbf x(s), d\bm{\lambda}(s)]^T$, $\mathbf X_0=[\mathbf x_0,\bm{\lambda}_0]^T$, $\mathbf X(s)=[\mathbf x(s),\bm{\lambda}(s)]^T$, $\hat{\bm{\sigma}}=[\bm{\sigma},\bm 0]^T$ and $d\hat{\mathbf B}(s)=[d\mathbf B(s),d\mathbf B_{\bm{\lambda}}(s)]^T$. Then we get the following equation
\begin{align}\label{sde}
d\mathbf X(s)=\hat{\mathbf K}\mathbf X_0ds+\mathbf A\mathbf X(s)ds+\hat{\bm{\sigma}}d\hat{\mathbf B}(s).
\end{align}
Following \cite{oksendal2003} we get a unique solution of the stochastic differential equation expressed in Equation (\ref{sde}) as 
\begin{align}\label{sde0}
\mathbf X(s)=\exp(\mathbf A s)\left[\hat{\mathbf K}\mathbf X_0+\exp(-\mathbf As)\hat{\bm{\sigma}}\hat{\mathbf B}(s)+\int_0^t\exp(-\mathbf As)\mathbf A\hat{\bm{\sigma}}\hat{\mathbf B}(s)ds\right].
\end{align}

\section{Stochastic differential games with an explicit feedback Nash equilibrium}

Propositions \ref{p0} and \ref{p1} states that, for agent $i$ and given $h^i(s,x^i)$ one can get a optimal Nash feedback control $\phi^{i*}(s,x^i)$ and for a unique solution of the transition wave function the unique opinion dynamics is $x^{i*}$. In this section I am considering two main consensus: full consensus or complete information and consensus under a leader who can influence other agents' opinions.

First, consider the consensus under complete information. Let there be a network where all agents are connected with each other or $\eta_i=N\setminus\{i\}$. As every agent has equal power to influence others, in the long run a consensus will be eventually reached. As some agents are stubborn, their opinions might not be influenced by others and a full consensus is not reached. Following \cite{niazi2016} assume all the parameters of agent $i$'s cost function are equal or $k_i=k$, $w_{ij}=w_{ji}=w$ for all $i\in N$ and $(i,j)\in E$ where agent $i$'s stochastic opinion dynamics is represented by
\begin{align}\label{fc}
dx^i(s)=\left[\mbox{$\frac{1}{n}$}\sum_{j=1}^nx^{j*}+\gamma(s)\left(x^i(s)-\mbox{$\frac{1}{n}$}\sum_{j=1}^nx^{j*}\right)-u^i(s)\right]ds+\sqrt{2\sigma}dB^i(s),
\end{align}
where $\gamma(s)=\frac{k}{\lambda_1}+\left(\frac{nw}{\lambda_1}\right)\frac{cosh\left[\sqrt{\lambda_1}(t-s)\right]}{cosh\left(\sqrt{\lambda_1}t\right)}$, $\lambda_1=k+nw$ and $\sigma$ is a constant diffusion component. In Equation (\ref{fc}) $x^{j*}$ is the optimal opinion of $j^{th}$ agent according to agent $i$ because, under complete information agent $i$ has the information of all possible reaction functions of agent $j$ but does not know what reaction function agent $j$ will play. Therefore, agent $i$ assumes agent $j$ is rational and calculates optimal opinion $x^{j*}$. Opinion trajectory explained in Equation (\ref{fc}) has drift part and a diffusion part. The drift part has three components, the first component is the average of optimal opinions of all the agents in the network, the second term depends on the difference between the opinion of agent $i$ at time $s$ and the average and the third component is the control of agent $i$. As control is the cost of agent $i$ in the opinion dynamics, it comes with a negative sign at the front. I do not consider other agents' controls in Equation (\ref{fc}) because, I assume all of the agents' control in this network are independent to each other.

\begin{prop}\label{p2}
	Suppose agent $i$ minimizes the objective cost function 
	\[
	\int_0^t\E_0\  \bigg\{\mbox{$\frac{1}{2}$}nw\left[x^i(s)-x^j(s)\right]^2+\mbox{$\frac{1}{2}$}k\left[x^i(s)-x_0^i\right]^2+\mbox{$\frac{1}{2}$}\left[u^i(s)\right]^2\bigg |\mathcal F_0^x\bigg\} ds,
	\]
	subject to the stochastic opinion dynamics expressed in Equation (\ref{fc}). For $b,d>0$, define $h^i(s,x^i)=\exp(sbx^i+d)$.
	
	(i) Then for 
	\begin{align}
	f^i(s,\mathbf x,\lambda^i,u^i)&=\mbox{$\frac{1}{2}$}nw\left(x^i-x^j\right)^2+\mbox{$\frac{1}{2}$}k\left(x^i-x_0^i\right)^2+\mbox{$\frac{1}{2}$}\left(u^i\right)^2+b\lambda^ix^ih^i(s,x^i)+\mbox{$\frac{\partial \lambda^i}{\partial s}$}h^i(s,x^i)\notag\\&\hspace{.5cm}+sb\lambda^ih^i(s,x^i)\left[\mbox{$\frac{1}{n}$}\sum_{j=1}^nx^{j*}+\gamma\left(x^i-\mbox{$\frac{1}{n}$}\sum_{j=1}^nx^{j*}\right)-u^i\right]+s^2b^2\sigma\lambda^ih^i(s,x^i),\notag
	\end{align}
	a feedback Nash Equilibrium control of opinion dynamics
	\begin{align}\label{fc0}
	\phi^{i*}(s,x^i)=p+\left\{q+\left[q^2+(r-p^2)^3\right]^{\frac{1}{2}}\right\}^{\frac{1}{3}}+\left\{q-\left[q^2+(r-p^2)^3\right]^{\frac{1}{2}}\right\}^{\frac{1}{3}},
	\end{align}
	where
	\begin{align}
	p&=-\frac{B_2(s,\gamma,x^i,x^j,\lambda^i)}{3B_1(s,x^i,\lambda^i)},\notag\\q&=[p(s,\gamma,x^i,x^j,\lambda^i)]^3+\frac{B_2(s,\gamma,x^i,x^j,\lambda^i)B_3(s,\gamma,x^i,x^j,\lambda^i)-3B_1(s,x^i,\lambda^i)B_4(s,\gamma,x^i,x^j,\lambda^i)}{6[B_1(s,x^i,\lambda^i)]^2},\notag\\
	r &=\frac{B_3(s,\gamma,x^i,x^j,\lambda^i)}{3B_1(s,x^i,\lambda^i)},\notag
	\end{align}
	$B_1=(C_2)^2$, $B_2=-C_2(2A_2+C_1)$, $B_3=(A_2)^2-2A_2C_1C_2-(C_3)^2$, $B_4=A_1C_3-C_1(A_2)^2$, $C_1=sb\lambda^ih^i(s,x^i)$, $C_2=(sb)^3\lambda^ih^i(s,x^i)$, and $C_3=(sb)^2\lambda^ih^i(s,x^i)$.
	
	(ii) For a unique solution of the wave function $\Psi_{is}(x^i)$ as expressed in Proposition \ref{p1} and $\lambda^i$ is a $C^2$ function with respect to $s$, an optimal opinion ${x}^{i*}$ is obtained by solving following equation
	\begin{align}
	& h^i(s,x^i)\biggr\{2b\lambda^ix^i+sb^3\lambda^i(x^i)^2+sb\mbox{$\frac{\partial^2\lambda^i}{\partial s^2}$}+b(1+sbx^i)\mbox{$\frac{\partial \lambda^i}{\partial s}$}+\left[[(sb)^2+b(1+b+sb)]\mbox{$\frac{\partial\lambda^i}{\partial s}$}\right]\notag\\&\hspace{1cm}\left[\mbox{$\frac{1}{n}$}\sum_{j=1}^nx^{j*}+\gamma\left(x^i-\mbox{$\frac{1}{n}$}\sum_{j=1}^nx^{j*}\right)-u^i\right]+\gamma\left[sb\mbox{$\frac{\partial \lambda^i}{\partial s}$}+b\lambda^i(1+sx^i)\right]\notag\\&\hspace{2cm}+sb\lambda^i\left[1+sb\left(x^i-\mbox{$\frac{1}{n}$}\sum_{j=1}^nx^{j*}\right)\right]\mbox{$\frac{\partial\gamma}{\partial s}$}+\sigma s^2b^2\bigg[\lambda^i(3+sbx^i)+\mbox{$\frac{\partial\lambda^i}{\partial s}$}\bigg]\biggr\}\notag\\&=x^i(k+nw)-(nwx^j+kx_0^i)+bh^i(s,x^i)\biggr\{sb\lambda^ix^i(1+s\gamma)+\lambda^i+s\mbox{$\frac{\partial \lambda^i}{\partial s}$}\notag\\&\hspace{1cm}+s^2b\lambda^i\left((1-\gamma)\mbox{$\frac{1}{n}$}\sum_{j=1}^nx^{j*}-u^i\right)+s\lambda^i(\gamma+s^2b\sigma)\biggr\},
	\end{align}
	which is
	\begin{align}\label{op6.0}
	x^{i*}&=A_{11}+\left\{A_{12}+\left[(A_{12})^2+[A_{13}-(A_{11})^2]^3\right]^{\frac{1}{2}}\right\}^{\frac{1}{3}}+\left\{A_{12}-\left[(A_{12})^2+[A_{13}-(A_{11})^2]^3\right]^{\frac{1}{2}}\right\}^{\frac{1}{3}},
	\end{align} 
	where
	\begin{align}
	& A_{13}=\frac{A_9(s,\sigma,\gamma,\lambda^i,u^i,x^j)}{3A_7(s,\lambda^i)}\notag\\
	& A_{12}=(A_{11})^3+\frac{A_8(s,\sigma,\gamma,\lambda^i,u^i,x^j)A_9(s,\sigma,\gamma,\lambda^i,u^i,x^j)-3A_7(s,\lambda^i)A_{10}(s,\sigma,\gamma,\lambda^i,u^i,x^j)}{6[A_7(s,\lambda^i)]^2},\notag\\ 
	& A_{11}=-\ \frac{A_8(s,\sigma,\gamma,\lambda^i,u^i,x^j)}{3A_7(s,\lambda^i)},\notag\\
	& A_{10}(s,\sigma,\gamma,\lambda^i,u^i,x^j)=A_3(s,\sigma,\gamma,\lambda^i,u^i,x^j)+beA_5(s,\sigma,\gamma,\lambda^i,u^i,x^j),\notag\\
	& A_9(s,\sigma,\gamma,\lambda^i,u^i,x^j)=beA_6(s,\sigma,\gamma,\lambda^i,u^i,x^j)+sb^2A_5(s,\sigma,\gamma,\lambda^i,u^i,x^j)-(k+nw),\notag\\
	& A_8(s,\sigma,\gamma,\lambda^i,u^i,x^j)=sb^2[e\lambda^i+A_6(s,\sigma,\gamma,\lambda^i,u^i,x^j)],\notag\\
	& A_7(s,\lambda^i)=s^2b^4\lambda^i,\notag\\
	& A_6(s,\sigma,\gamma,\lambda^i,u^i,x^j)=[2+\gamma s+s^2b\mbox{$\frac{\partial}{\partial s}$}\gamma+\sigma(sb)^2-sb(1+s\gamma)]\lambda^i+[sb+\gamma(1+b+sb+s^2b)]\mbox{$\frac{\partial\lambda^i}{\partial s}$},\notag 
	\end{align}
	and
	\begin{align}
	&
	A_5(s,\sigma,\gamma,\lambda^i,u^i,x^j)\notag\\&=s\mbox{$\frac{\partial^2\lambda^i}{\partial s^2}$}+\mbox{$\frac{\partial\lambda^i}{\partial s}$}+\left[(1+b+sb+s^2b)\mbox{$\frac{\partial\lambda^i}{\partial s}$}\right]\left[(1-\gamma)\mbox{$\frac{1}{n}$}\sum_{j=1}^nx^{j*}-u^i\right]+\gamma\lambda^i\left(1+s\mbox{$\frac{\partial\lambda^i}{\partial s}$}\right)\notag\\&\hspace{.5cm}+s\lambda^i\left[1-sb\mbox{$\frac{1}{n}$}\sum_{j=1}^nx^{j*}\right]\mbox{$\frac{\partial\gamma}{\partial s}$}+\sigma s^2b^2\big(3\lambda^i+\mbox{$\frac{\partial\lambda^i}{\partial s}$}\big)-A_4(s,\sigma,\gamma,\lambda^i,u^i,x^j).\notag
	\end{align}
	
	(iii) The opinion difference between agents $i$ and $j$ at time $s\in[0,t]$ is
	\begin{align}
	|\Delta x^{ij}(s)|\leq|\Delta x_0^{ij}|+\left|\int_0^t\left[\gamma(s)\Delta x^{ij}(s)-\Delta u^{ij}(s)\right]ds\right|+\left|\sqrt{2\sigma}\right|\left|\int_0^t[dB^i(s)-dB^j(s)]\right|,\notag
	\end{align}
	where $\Delta x^{ij}(s)=x^i(s)-x^j(s)$, $\Delta x_0^{ij}=x_0^i-x_0^j$ and $\Delta u^{ij}(s)=u^i(s)-u^j(s)$.
\end{prop}

\begin{proof}
	(i). Let $h^i(s,x^i)=\exp(sbx^i+d)$, for a finite $b>0$ and $d>0$ with $\frac{\partial}{\partial s}h^i(s,x^i)=bx^ih^i(s,x^i)$, $\frac{\partial}{\partial x^i}h^i(s,x^i)=sbh^i(s,x^i)$ and $\frac{\partial^2}{\partial (x^i)^2}h^i(s,x^i)=s^2b^2h^i(s,x^i)$. Hence, Proposition \ref{p0} implies,
	\begin{align}\label{fc1}
	f^i(s,\mathbf x,\lambda^i,u^i)&=\mbox{$\frac{1}{2}$}nw\left(x^i-x^j\right)^2+\mbox{$\frac{1}{2}$}k\left(x^i-x_0^i\right)^2+\mbox{$\frac{1}{2}$}\left(u^i\right)^2+b\lambda^ix^ih^i(s,x^i)+\mbox{$\frac{\partial}{\partial s}$}\lambda^ih^i(s,x^i)\notag\\&\hspace{.5cm}+sb\lambda^ih^i(s,x^i)\left[\mbox{$\frac{1}{n}$}\sum_{j=1}^nx^{j*}+\gamma\left(x^i-\mbox{$\frac{1}{n}$}\sum_{j=1}^nx^{j*}\right)-u^i\right]+s^2b^2\sigma\lambda^ih^i(s,x^i).
	\end{align}
	Now
	\begin{align}\label{fc2}
	\mbox{$\frac{\partial}{\partial x^i}$}f^i(s,\mathbf x,\lambda^i,u^i)&=nw(x^i-x^j)+k(x^i-x_0^i)+bh^i(s,x^i)\biggr\{s\mbox{$\frac{\partial\lambda^i}{\partial s}$}\notag\\&\hspace{.5cm}+\lambda^i\left[1+bsx^i+s^2b\left[\mbox{$\frac{1}{n}$}\sum_{j=1}^nx^{j*}+\gamma\left(x^i-\mbox{$\frac{1}{n}$}\sum_{j=1}^nx^{j*}\right)-u^i\right]+s\gamma+\sigma s^3b^2\right]\biggr\}\notag\\&=A_1(s,\gamma,x^i,x^j,\lambda^i)-s^2b^2\lambda^ih^i(s,x^i)u^i,\notag\\
	\mbox{$\frac{\partial^2}{\partial (x^i)^2}$}f^i(s,\mathbf x,\lambda^i,u^i)&=nw+k+sb^2h^i(s,x^i)\biggr\{s\mbox{$\frac{\partial\lambda^i}{\partial s}$}+\lambda^i\biggr[1+sbx^i+s\gamma+\sigma s^3b^2\notag\\&\hspace{.5cm}+s^2b\bigg[\mbox{$\frac{1}{n}$}\sum_{j=1}^nx^{j*}+\gamma\left(x^i-\mbox{$\frac{1}{n}$}\sum_{j=1}^nx^{j*}\right)-u^i\bigg]\biggr]\biggr\}+sb^2\lambda^i(1+s\gamma)h^i(s,x^i)\notag\\&=A_2(s,\gamma,x^i,x^j,\lambda^i)-s^3b^3\lambda^ih^i(s,x^i)u^i,\notag\\
	\mbox{$\frac{\partial}{\partial u^i}$}f^i(s,\mathbf x,\lambda^i,u^i)&=u^i-sb\lambda^ih^i(s,x^i),
	\end{align}
	and,
	\begin{align}\label{fc3}
	\mbox{$\frac{\partial^2}{\partial x^i\partial u^i}$}f^i(s,\mathbf x,\lambda^i,u^i)&=-s^2b^2\lambda^i h^i(s,x^i).
	\end{align}
	Therefore, Equation (\ref{w21}) implies
	\begin{align}
	&\left[u^i-sb\lambda^ih^i(s,x^i)\right]\left[A_2(s,\gamma,x^i,x^j,\lambda^i)-s^3b^3u^i\lambda^ih^i(s,x^i)\right]^2\notag\\&=2s^2b^2\lambda^ih^i(s,x^i)\left[s^2b^2u^i\lambda^ih^i(s,x^i)-A_1(s,\gamma,x^i,x^j,\lambda^i)\right],\notag
	\end{align}
	and we get a cubic polynomial with respect to control
	\begin{align}\label{fc4}
	B_1(s,x^i,\lambda^i) (u^i)^3+B_2(s,\gamma,x^i,x^j,\lambda^i)(u^i)^2+B_3(s,\gamma,x^i,x^j,\lambda^i)u^i+B_4(s,\gamma,x^i,x^j,\lambda^i)=0,
	\end{align}
	where $B_1=(C_2)^2$, $B_2=-C_2(2A_2+C_1)$, $B_3=(A_2)^2-2A_2C_1C_2-(C_3)^2$, $B_4=A_1C_3-C_1(A_2)^2$, $C_1(s,x^i,\lambda^i)=sb\lambda^ih^i(s,x^i)$, $C_2(s,x^i,\lambda^i)=(sb)^3\lambda^ih^i(s,x^i)$, and 
	$C_3(s,x^i,\lambda^i)=(sb)^2\lambda^ih^i(s,x^i)$.
	Therefore, Equation (\ref{fc4}) gives feedback Nash equilibrium of control
	\begin{align}
	\phi^{i*}=p+\left\{q+\left[q^2+(r-p^2)^3\right]^{\frac{1}{2}}\right\}^{\frac{1}{3}}+\left\{q-\left[q^2+(r-p^2)^3\right]^{\frac{1}{2}}\right\}^{\frac{1}{3}},
	\end{align}
	where
	\begin{align}
	& p=-\frac{B_2(s,\gamma,x^i,x^j,\lambda^i)}{3B_1(s,x^i,\lambda^i)},\notag\\& q=p^3+\frac{B_2(s,\gamma,x^i,x^j,\lambda^i)B_3(s,\gamma,x^i,x^j,\lambda^i)-3B_1(s,x^i,\lambda^i)B_4(s,\gamma,x^i,x^j,\lambda^i)}{6[B_1(s,x^i,\lambda^i)]^2},\notag
	\end{align}
	and
	\begin{align}
	r=\frac{B_3(s,\gamma,x^i,x^j,\lambda^i)}{3B_1(s,x^i,\lambda^i)}.\notag
	\end{align}
	
	(ii). In order to prove the second part let us use Proposition \ref{p1}. The right hand side of Equation (\ref{op0}) becomes,
	\begin{align}\label{op1}
	& \mbox{$\frac{\partial}{\partial x^i}$}f^i[s,\mathbf x(s),\lambda^i(s),u^i(s)]\notag\\&=x^i(k+nw)-(nwx^j+kx_0^i)+bh^i(s,x^i)\biggr\{sb\lambda^ix^i(1+s\gamma)+\lambda^i+s\mbox{$\frac{\partial\lambda^i}{\partial s}$}\notag\\&\hspace{1cm}+s^2b\lambda^i\left((1-\gamma)\mbox{$\frac{1}{n}$}\sum_{j=1}^nx^{j*}-u^i\right)+s\lambda^i(\gamma+s^2b\sigma)\biggr\}\notag\\&=x^i(k+nw)-A_3(w,k,x^j)+bh^i(s,x^i)\left[A_4(s,\sigma,\gamma,\lambda^i,u^i,x^j)+sb\lambda^ix^i(1+s\gamma)\right],
	\end{align}
	the left hand side implies
	\begin{align}
	&\mbox{$\frac{\partial}{\partial s}$}f^i[s,\mathbf x(s),\lambda^i(s),u^i(s)]\notag\\&=h^i(s,x^i)\biggr\{\lambda^i(bx^i)^2+\mbox{$\frac{\partial^2\lambda^i}{\partial s^2}$}+bx^i\mbox{$\frac{\partial \lambda^i}{\partial s}$}\notag\\&\hspace{.25cm}+b\left[s\mbox{$\frac{\partial\lambda^i}{\partial s}$}+\lambda^i(1+sx^i)\right]\left[\mbox{$\frac{1}{n}$}\sum_{j=1}x^{j*}+\gamma\left(x^i-\mbox{$\frac{1}{n}$}\sum_{j=1}x^{j*}\right)-u^i\right]\notag\\&\hspace{.5cm}+sb\lambda^i\left(x^i-\mbox{$\frac{1}{n}$}\sum_{j=1}x^{j*}\right)\mbox{$\frac{\partial\gamma}{\partial s}$}+sb^2\sigma\lambda^i(2b+sb^2x^i)+s^2b^2\sigma\mbox{$\frac{\partial\lambda^i}{\partial s}$}\biggr\},\notag
	\end{align}
	and
	\begin{align}\label{op2}
	& \mbox{$\frac{\partial^2}{\partial s \partial x^i}$}f^i[s,\mathbf x(s),\lambda^i(s),u^i(s)]\notag\\&=h^i(s,x^i)\biggr\{2b\lambda^ix^i+sb^3\lambda^i(x^i)^2+sb\mbox{$\frac{\partial^2\lambda^i}{\partial s^2}$}+b(1+sbx^i)\mbox{$\frac{\partial\lambda^i}{\partial s}$}+\left[[(sb)^2+b(1+b+sb)]\mbox{$\frac{\partial\lambda^i}{\partial s}$}\right]\times\notag\\&\hspace{1cm}\left[\mbox{$\frac{1}{n}$}\sum_{j=1}^nx^{j*}+\gamma\left(x^i-\mbox{$\frac{1}{n}$}\sum_{j=1}^nx^{j*}\right)-u^i\right]+\gamma\left[sb\mbox{$\frac{\partial\lambda^i}{\partial s}$}+b\lambda^i(1+sx^i)\right]\notag\\&\hspace{2cm}+sb\lambda^i\left[1+sb\left(x^i-\mbox{$\frac{1}{n}$}\sum_{j=1}^nx^{j*}\right)\right]\mbox{$\frac{\partial\gamma}{\partial s}$}+\sigma s^2b^3\bigg[\lambda^i(3+sbx^i)+\mbox{$\frac{\partial\lambda^i}{\partial s}$}\bigg]\biggr\}.
	\end{align}
	Matching Equations (\ref{op1}) and (\ref{op2}) we get,
	\begin{align}
	& h^i(s,x^i)\biggr\{2b\lambda^ix^i+sb^3\lambda^i(x^i)^2+sb\mbox{$\frac{\partial^2\lambda^i}{\partial s^2}$}+b(1+sbx^i)\mbox{$\frac{\partial\lambda^i}{\partial s}$}+\left[[(sb)^2+b(1+b+sb)]\mbox{$\frac{\partial\lambda^i}{\partial s}$}\right]\notag\\&\hspace{.25cm}\left[\mbox{$\frac{1}{n}$}\sum_{j=1}^nx^{j*}+\gamma\left(x^i-\mbox{$\frac{1}{n}$}\sum_{j=1}^nx^{j*}\right)-u^i\right]+\gamma\left[sb\mbox{$\frac{\partial\lambda^i}{\partial s}$}+b\lambda^i(1+sx^i)\right]\notag\\&\hspace{.5cm}+sb\lambda^i\left[1+sb\left(x^i-\mbox{$\frac{1}{n}$}\sum_{j=1}^nx^{j*}\right)\right]\mbox{$\frac{\partial\gamma}{\partial s}$}+\sigma s^2b^3\bigg[\lambda^i(3+sbx^i)+\mbox{$\frac{\partial\lambda^i}{\partial s}$}\bigg]\biggr\}\notag\\&=x^i(k+nw)-(nwx^j+kx_0^i)+bh^i(s,x^i)\biggr\{sb\lambda^ix^i(1+s\gamma)+\lambda^i+s\mbox{$\frac{\partial\lambda^i}{\partial s}$}\notag\\&\hspace{1cm}+s^2b\lambda^i\left((1-\gamma)\mbox{$\frac{1}{n}$}\sum_{j=1}^nx^{j*}-u^i\right)+s\lambda^i(\gamma+s^2b\sigma)\biggr\},\notag
	\end{align}
	or,
	\begin{align}\label{op3}
	& bh^i(s,x^i)\biggr\{2\lambda^ix^i+sb^2\lambda^i(x^i)^2+s\mbox{$\frac{\partial^2\lambda^i}{\partial s^2}$}+(1+sbx^i)\mbox{$\frac{\partial\lambda^i}{\partial s}$}+\left[(1+s^2b+b+sb)\mbox{$\frac{\partial\lambda^i}{\partial s}$}\right]\times\notag\\&\hspace{.25cm}\left[\mbox{$\frac{1}{n}$}\sum_{j=1}^nx^{j*}+\gamma\left(x^i-\mbox{$\frac{1}{n}$}\sum_{j=1}^nx^{j*}\right)-u^i\right]+\gamma\left[s\mbox{$\frac{\partial\lambda^i}{\partial s}$}+\lambda^i(1+sx^i)\right]\notag\\&\hspace{.5cm}+s\lambda^i\left[1+sb\left(x^i-\mbox{$\frac{1}{n}$}\sum_{j=1}^nx^{j*}\right)\right]\mbox{$\frac{\partial\gamma}{\partial s}$}+\sigma s^2b^2\bigg[\lambda^i(3+sbx^i)+\mbox{$\frac{\partial\lambda^i}{\partial s}$}\bigg]\notag\\&\hspace{1cm}-\bigg[sb\lambda^ix^i(1+s\gamma)+\lambda^i+s\mbox{$\frac{\partial}{\partial s}$}\lambda^i+s^2b\lambda^i\left((1-\gamma)\mbox{$\frac{1}{n}$}\sum_{j=1}^nx^{j*}-u^i\right)+s\lambda^i(\gamma+s^2b\sigma)\bigg]\biggr\}\notag\\&=x^i(k+nw)-A_3(w,k,x^j).
	\end{align}
	As $h^i(s,x^i)=\exp(sbx^i+d)$, for $b>0$, $d>0$ fixed and a very small value of $x^i$ it can be approximated as $h^i(s,x^i)=1+(sbx^i+d)+o([sbx^i+d]^2)\approx1+d+sbx^i=e+sbx^i$ where assume $e=1+d$. 
	
	Therefore,
	\begin{align}\label{op4}
	& (be+sb^2x^i)\biggr\{2\lambda^ix^i+sb^2\lambda^i(x^i)^2+s\mbox{$\frac{\partial^2\lambda^i}{\partial s^2}$}+(1+sbx^i)\mbox{$\frac{\partial\lambda^i}{\partial s}$}+\left[(1+s^2b+b+sb)\mbox{$\frac{\partial\lambda^i}{\partial s}$}\right]\times\notag\\&\hspace{.25cm}\left[\mbox{$\frac{1}{n}$}\sum_{j=1}^nx^{j*}+\gamma\left(x^i-\mbox{$\frac{1}{n}$}\sum_{j=1}^nx^{j*}\right)-u^i\right]+\gamma\left[s\mbox{$\frac{\partial\lambda^i}{\partial s}$}+\lambda^i(1+sx^i)\right]\notag\\&\hspace{.5cm}+s\lambda^i\left[1+sb\left(x^i-\mbox{$\frac{1}{n}$}\sum_{j=1}^nx^{j*}\right)\right]\mbox{$\frac{\partial\gamma}{\partial s}$}+\sigma s^2b^2\bigg[\lambda^i(3+sbx^i)+\mbox{$\frac{\partial\lambda^i}{\partial s}$}\bigg]\notag\\&\hspace{1cm}-\bigg[sb\lambda^ix^i(1+s\gamma)+\lambda^i+s\mbox{$\frac{\partial\lambda^i}{\partial s}$}+s^2b\lambda^i\left((1-\gamma)\mbox{$\frac{1}{n}$}\sum_{j=1}^nx^{j*}-u^i\right)+s\lambda^i(\gamma+s^2b\sigma)\bigg]\biggr\}\notag\\&=x^i(k+nw)-A_3(w,k,x^j).
	\end{align}
	After rearranging terms of Equation (\ref{op4}) we get a cubic polynomial opinion of agent $i$
	\begin{align}\label{op5}
	A_7(s,\lambda^i) (x^i)^3+A_8(s,\sigma,\gamma,\lambda^i,u^i,x^j) (x^i)^2+A_9(s,\sigma,\gamma,\lambda^i,u^i,x^j) x^i+A_{10}(s,\sigma,\gamma,\lambda^i,u^i,x^j)=0,
	\end{align}
	where
	\begin{align}
	& A_{10}(s,\sigma,\gamma,\lambda^i,u^i,x^j)=A_3(s,\sigma,\gamma,\lambda^i,u^i,x^j)+beA_5(s,\sigma,\gamma,\lambda^i,u^i,x^j),\notag\\
	& A_9(s,\sigma,\gamma,\lambda^i,u^i,x^j)=beA_6(s,\sigma,\gamma,\lambda^i,u^i,x^j)+sb^2A_5(s,\sigma,\gamma,\lambda^i,u^i,x^j)-(k+nw),\notag\\
	& A_8(s,\sigma,\gamma,\lambda^i,u^i,x^j)=sb^2[e\lambda^i+A_6(s,\sigma,\gamma,\lambda^i,u^i,x^j)],\notag\\
	& A_7(s,\lambda^i)=s^2b^4\lambda^i,\notag\\
	& A_6(s,\sigma,\gamma,\lambda^i,u^i,x^j)=[2+\gamma s+s^2b\mbox{$\frac{\partial\gamma}{\partial s}$}+\sigma(sb)^2-sb(1+s\gamma)]\lambda^i+[sb+\gamma(1+b+sb+s^2b)]\mbox{$\frac{\partial\lambda^i}{\partial s}$},\notag
	\end{align}
	and
	\begin{align}
	& A_5(s,\sigma,\gamma,\lambda^i,u^i,x^j)\notag\\&=s\mbox{$\frac{\partial^2\lambda^i}{\partial s^2}$}+\mbox{$\frac{\partial\lambda^i}{\partial s}$}+\left[(1+b+sb+s^2b)\mbox{$\frac{\partial\lambda^i}{\partial s}$}\right]\left[(1-\gamma)\mbox{$\frac{1}{n}$}\sum_{j=1}^nx^{j*}-u^i\right]+\gamma\lambda^i\left(1+s\mbox{$\frac{\partial\lambda^i}{\partial s}$}\right)\notag\\&\hspace{.5cm}+s\lambda^i\left[1-sb\mbox{$\frac{1}{n}$}\sum_{j=1}^nx^{j*}\right]\mbox{$\frac{\partial\gamma}{\partial s}$}+\sigma s^2b^2\big(3\lambda^i+\mbox{$\frac{\partial\lambda^i}{\partial s}$}\big)-A_4(s,\sigma,\gamma,\lambda^i,u^i,x^j).\notag
	\end{align}
	After solving Equation (\ref{op5}) we get a set of optimal opinions for agent $i$
	\begin{align}\label{op6}
	x^{i*}&=A_{11}+\left\{A_{12}+\left[(A_{12})^2+[A_{13}-(A_{11})^2]^3\right]^{\frac{1}{2}}\right\}^{\frac{1}{3}}+\left\{A_{12}-\left[(A_{12})^2+[A_{13}-(A_{11})^2]^3\right]^{\frac{1}{2}}\right\}^{\frac{1}{3}},
	\end{align}
	where 
	\begin{align}
	& A_{11}=-\ \frac{A_8(s,\sigma,\gamma,\lambda^i,u^i,x^j)}{3A_7(s,\lambda^i)},\notag\\& A_{12}=(A_{11})^3+\frac{A_8(s,\sigma,\gamma,\lambda^i,u^i,x^j)A_9(s,\sigma,\gamma,\lambda^i,u^i,x^j)-3A_7(s,\lambda^i)A_{10}(s,\sigma,\gamma,\lambda^i,u^i,x^j)}{6[A_7(s,\lambda^i)]^2},\notag
	\end{align}
	and
	\begin{align}
	A_{13}=\frac{A_9(s,\sigma,\gamma,\lambda^i,u^i,x^j)}{3A_7(s,\lambda^i)}.\notag
	\end{align}
	
	(iii). The integral forms of opinions of agents $i$ and $j$ obtained from the Equation (\ref{fc}) are
	\begin{align}
	x^i(s)=x_0^i+\int_0^t\left[\mbox{$\frac{1}{n}$}\sum_{j=1}^nx^{j*}+\gamma(s)\left(x^i(s)-\mbox{$\frac{1}{n}$}\sum_{j=1}^nx^{j*}\right)-u^i(s)\right]ds+\sqrt{2\sigma}\int_0^tdB^i(s),\notag
	\end{align}
	and
	\begin{align}
	x^j(s)=x_0^j+\int_0^t\left[\mbox{$\frac{1}{n}$}\sum_{i=1}^nx^{i*}+\gamma(s)\left(x^j(s)-\mbox{$\frac{1}{n}$}\sum_{i=1}^nx^{i*}\right)-u^j(s)\right]ds+\sqrt{2\sigma}\int_0^tdB^j(s),\notag
	\end{align}
	where $x_0^i$ and $x_0^j$ are the initial opinions of agents $i$ and $j$. As agents $i$ and $j$ comes from the same population hence, $\frac{1}{n}\sum_{i=1}^nx^{i*}=\frac{1}{n}\sum_{j=1}^nx^{j*}$. Subtracting $x^j(s)$ from $x^i(s)$ gives,
	\begin{align}
	& x^i(s)-x^j(s)\notag\\&=(x_0^i-x_0^j)+\int_0^t\left[\gamma(s)[x^i(s)-x^j(s)]-[u^i(s)-u^j(s)\right]ds+\sqrt{2\sigma}\int_0^t\left[dB^i(s)-dB^j(s)\right]\notag
	\end{align}
	and taking absolute value on both sides and using triangle inequality we get,
	\begin{align}
	|\Delta x^{ij}(s)|\leq|\Delta x_0^{ij}|+\left|\int_0^t\left[\gamma(s)\Delta x^{ij}(s)-\Delta u^{ij}(s)\right]ds\right|+\left|\sqrt{2\sigma}\right|\left|\int_0^t[dB^i(s)-dB^j(s)]\right|,\notag
	\end{align}
	where $\Delta x^{ij}(s)=x^i(s)-x^j(s)$, $\Delta x_0^{ij}=x_0^i-x_0^j$ and $\Delta u^{ij}(s)=u^i(s)-u^j(s)$.  	
\end{proof}	

Consider the consensus with a leader (agent $1$) under complete information. It might be a network where agent $1$, the political analyst who can influence the decision of the rest of the agents but not the other way. Furthermore, I also assume that, before a game starts the leader makes their optimal opinion based on the history of the network and their perspective of opinion performance of other agents. Once agent $1$ optimizes their opinion at the beginning of the game, they never change their mind and influences in other agents' decisions. Therefore, leader's cost functional is defined as,
\begin{align}\label{l0}
L^1(s,\mathbf x,x_0^1,u^1)&=\int_0^t \mbox{$\frac{1}{2}$} \bigg(n\bar w\left[x^1(s)-\tilde{x}^j(s)\right]^2+k_1\left[x^1(s)-x_0^1\right]^2+\left[u^1(s)\right]^2\bigg)\ ds,
\end{align}
where $\bar w\in[0,\infty)$ is a parameter assigned by agent $1$ to weight the susceptibility of agent $j$ to influence them before the game starts, $k_1$ is a finite positive constant which measures the stubbornness of the leader, $u^1$ is the opinion control and $\tilde x^j<x^{j*}$ be the fixed opinion values of the other agents according to agent $1$. The reason behind the assumption $\tilde x^j<x^{j*}$ is that, the leader is a rational person and they want to get more return out of this network than any other agent and assigns an opinion $\tilde x^j$ which is less than agent $j$'s optimal opinion before a game starts. Opinion dynamics of the leader (agent $1$) is 
\begin{align}\label{l1}
dx^1(s)=\left[\mbox{$\frac{1}{n-1}$}\sum_{j=2}^n\tilde{x}^{j}+\hat{\gamma}(s)\left(x^1(s)-\mbox{$\frac{1}{n-1}$}\sum_{j=2}^n\tilde{x}^{j}\right)-u^1(s)\right]ds+\sqrt{2\sigma^1}dB^1(s),
\end{align}
where $\hat{\gamma}(s)=\frac{k_1}{\hat {\lambda}_1}+\left(\frac{n\bar w}{\hat\lambda_1}\right)\frac{cosh\left[\sqrt{\hat\lambda_1}(t-s)\right]}{cosh\left(\sqrt{\hat\lambda_1}t\right)}$, $\hat\lambda_1=k_1+n\bar w$ and $\sigma^1$ is a constant diffusion component of the leader. Therefore a leader's problem is to minimize the expected cost functional $\E(L^1)$ with respect to their control $u^1$ and opinion $x^1$  subject to the Equation (\ref{l1}). Proposition \ref{p2} implies,

\begin{cor}\label{c1}
	Suppose the leader (agent $1$) has the objective cost function 
	\[
	\E\  \bigg\{\mbox{$\frac{1}{2}$}n\bar w\left[x^1(s)-\tilde{x}^j(s)\right]^2+\mbox{$\frac{1}{2}$}k_1\left[x^1(s)-x_0^1\right]^2+\mbox{$\frac{1}{2}$}\left[u^1(s)\right]^2\bigg\}
	\]
	subject to the stochastic opinion dynamics expressed in Equation (\ref{l1}). For $b,d>0$, define $h^1(s,x^1)=\exp(sbx^1+d)$.
	
	(i) Then for 
	\begin{align}
	&f^1(s,\mathbf x,\lambda^1,u^1)=\mbox{$\frac{1}{2}$}n\bar w\left(x^1-\tilde{x}^j\right)^2+\mbox{$\frac{1}{2}$}k_1\left(x^1-x_0^1\right)^2+\mbox{$\frac{1}{2}$}\left(u^1\right)^2+b\lambda^1x^1h^1(s,x^1)+\mbox{$\frac{\partial \lambda^1}{\partial s}$}h^1(s,x^1)\notag\\&\hspace{.5cm}+sb\lambda^1h^1(s,x^1)\left[\mbox{$\frac{1}{n-1}$}\sum_{j=2}^n\tilde{x}^{j}+\hat\gamma\left(x^1-\mbox{$\frac{1}{n-1}$}\sum_{j=2}^n\tilde{x}^{j}\right)-u^1\right]+s^2b^2\sigma^1\lambda^1h^1(s,x^1),\notag
	\end{align}
	an optimal control of the leader
	\begin{align}\label{l2}
	\hat{\phi}^{1*}(s,x^1)=\hat p+\left\{\hat q+\left[{\hat q}^2+(\hat r-{\hat p}^2)^3\right]^{\frac{1}{2}}\right\}^{\frac{1}{3}}+\left\{\hat q-\left[{\hat q}^2+(\hat r-{\hat p}^2)^3\right]^{\frac{1}{2}}\right\}^{\frac{1}{3}},
	\end{align}
	where
	\begin{align}
	\hat p&=-\frac{\hat{B}_2(s,\hat \gamma,x^1,\tilde{x}^j,\lambda^1)}{3\hat{B}_1(s,x^1,\lambda^1)},\notag\\\hat q&=(\hat p)^3+\frac{\hat{B}_2(s,\hat\gamma,x^1,\tilde{x}^j,\lambda^1)\hat{B}_3(s,\hat\gamma,x^1,\tilde{x}^j,\lambda^1)-3\hat{B}_1(s,x^1,\lambda^1)\hat{B}_4(s,\hat\gamma,x^1,\tilde{x}^j,\lambda^1)}{6[\hat{B}_1(s,x^1,\lambda^1)]^2},\notag\\
	\hat r &=\frac{\hat{B}_3(s,\hat\gamma,x^1,\tilde{x}^j,\lambda^1)}{3\hat{B}_1(s,x^1,\hat\lambda^1)},\notag
	\end{align}
	$\hat{B}_1=(\hat{C}_2)^2$, $\hat{B}_2=-\hat{C}_2(2\hat{A}_2+\hat{C}_1)$, $\hat{B}_3=(\hat{A}_2)^2-2\hat{A}_2\hat{C}_1\hat{C}_2-(\hat{C}_3)^2$, $\hat{B}_4=\hat{A}_1\hat{C}_3-\hat{C}_1(\hat{A}_2)^2$, $\hat{C}_1=sb\lambda^1h^1(s,x^1)$, $\hat{C}_2=(sb)^3\lambda^1h^1(s,x^1)$, and $\hat{C}_3=(sb)^2\lambda^1h^1(s,x^1)$.
	
	(ii) For a unique solution of the leader's  wave function $\Psi_{1s}(x)$ and $\lambda^1$ is a $C^2$ function with respect to $s$, a leader's optimal opinion ${x}^{1*}$ is obtained by solving following equation
	\begin{align}
	& h^1(s,x^1)\biggr\{2b\lambda^1x^1+sb^3\lambda^1(x^1)^2+sb\mbox{$\frac{\partial^2\lambda^1}{\partial s^2}$}+b(1+sbx^1)\mbox{$\frac{\partial \lambda^1}{\partial s}$}+\left[[(sb)^2+b(1+b+sb)]\mbox{$\frac{\partial \lambda^1}{\partial s}$}\right]\notag\\&\hspace{1cm}\left[\mbox{$\frac{1}{n-1}$}\sum_{j=2}^n\tilde{x}^{j}+\hat\gamma\left(x^1-\mbox{$\frac{1}{n-1}$}\sum_{j=2}^n\tilde{x}^{j}\right)-u^1\right]+\hat\gamma\left[sb\mbox{$\frac{\partial \lambda^1}{\partial s}$}+b\lambda^1(1+sx^1)\right]\notag\\&\hspace{2cm}+sb\lambda^1\left[1+sb\left(x^1-\mbox{$\frac{1}{n-1}$}\sum_{j=2}^n\tilde{x}^{j}\right)\right]\mbox{$\frac{\partial\hat{\gamma}}{\partial s}$}+\sigma^1 s^2b^3\bigg[\lambda^1(3+sbx^1)+\mbox{$\frac{\partial\lambda^1}{\partial s}$}\bigg]\biggr\}\notag\\&=x^1(k_1+n\bar w)-(n\bar w\tilde{x}^j+k_1x_0^1)+bh^1(s,x^1)\biggr\{sb\lambda^1x^1(1+s\hat\gamma)+\lambda^1+s\mbox{$\frac{\partial \lambda^1}{\partial s}$}\notag\\&\hspace{1cm}+s^2b\lambda^1\left((1-\hat\gamma)\mbox{$\frac{1}{n-1}$}\sum_{j=2}^n\tilde{x}^{j}-u^1\right)+s\lambda^1(\hat\gamma+s^2b\sigma^1)\biggr\},\notag
	\end{align}
	which is
	\begin{align}\label{l3}
	x^{1*}&=\hat{A}_{11}+\left\{\hat{A}_{12}+\left[(\hat{A}_{12})^2+[\hat{A}_{13}-(\hat{A}_{11})^2]^3\right]^{\frac{1}{2}}\right\}^{\frac{1}{3}}+\left\{\hat{A}_{12}-\left[(\hat{A}_{12})^2+[\hat{A}_{13}-(\hat{A}_{11})^2]^3\right]^{\frac{1}{2}}\right\}^{\frac{1}{3}},
	\end{align} 
	where
	\begin{align}
	& \hat{A}_{13}=\frac{\hat{A}_9(s,\sigma^1,\hat\gamma,\lambda^1,u^1,\tilde{x}^j)}{3\hat{A}_7(s,\lambda^1)}\notag\\
	& \hat{A}_{12}=(\hat{A}_{11})^3+\frac{\hat{A}_8(s,\sigma^1,\hat\gamma,\lambda^1,u^1,\tilde{x}^j)\hat{A}_9(s,\sigma^1,\hat\gamma,\lambda^1,u^1,\tilde{x}^j)-3\hat{A}_7(s,\lambda^1)\hat{A}_{10}(s,\sigma^1,\hat\gamma,\lambda^1,u^1,\tilde{x}^j)}{6[\hat{A}_7(s,\lambda^1)]^2},\notag\\ 
	& \hat{A}_{11}=-\ \frac{\hat{A}_8(s,\sigma^1,\hat\gamma,\lambda^1,u^1,\tilde{x}^j)}{3\hat{A}_7(s,\lambda^1)},\notag\\
	& \hat{A}_{10}(s,\sigma^1,\hat\gamma,\lambda^1,u^1,\tilde{x}^j)=\hat{A}_3(s,\sigma^1,\hat\gamma,\lambda^1,u^1,\tilde{x}^j)+be\hat{A}_5(s,\sigma^1,\hat\gamma,\lambda^1,u^1,\tilde{x}^j),\notag\\
	& \hat{A}_9(s,\sigma^1,\hat\gamma,\lambda^1,u^1,\tilde{x}^j)=be\hat{A}_6(s,\sigma^1,\hat\gamma,\lambda^1,u^1,\tilde{x}^j)+sb^2\hat{A}_5(s,\sigma^1,\hat\gamma,\lambda^1,u^1,\tilde{x}^j)-(k_1+n\bar w),\notag\\
	& \hat{A}_8(s,\sigma^1,\hat\gamma,\lambda^1,u^1,\tilde{x}^j)=sb^2[e\lambda^1+\hat{A}_6(s,\sigma^1,\hat\gamma,\lambda^1,u^1,\tilde{x}^j)],\notag\\
	& \hat{A}_7(s,\lambda^1)=s^2b^4\lambda^1,\notag\\
	& \hat{A}_6(s,\sigma^1,\gamma^1,\lambda^1,u^1,\tilde{x}^j)=[2+\hat\gamma s+s^2b\mbox{$\frac{\partial\hat{\gamma}}{\partial s}$}+\sigma^1(sb)^2-sb(1+s\hat\gamma)]\lambda^1+[sb+\hat\gamma(1+b+sb+s^2b)]\mbox{$\frac{\partial\lambda^1}{\partial s}$},\notag 
	\end{align}
	and
	\begin{align}
	&
	\hat{A}_5(s,\sigma^1,\hat\gamma,\lambda^1,u^1,\tilde{x}^j)\notag\\&=s\mbox{$\frac{\partial^2\lambda^1}{\partial s^2}$}+\mbox{$\frac{\partial\lambda^1}{\partial s}$}+\left[(1+b+sb+s^2b)\mbox{$\frac{\partial\lambda^1}{\partial s}$}\right]\left[(1-\hat\gamma)\mbox{$\frac{1}{n-1}$}\sum_{j=2}^n\tilde{x}^{j}-u^1\right]+\hat\gamma\lambda^1\left(1+s\mbox{$\frac{\partial\lambda^1}{\partial s}$}\right)\notag\\&\hspace{.5cm}+s\lambda^1\left[1-sb\mbox{$\frac{1}{n-1}$}\sum_{j=2}^n\tilde{x}^{j}\right]\mbox{$\frac{\partial\hat{\gamma}}{\partial s}$}+\sigma^1(sb)^2\big(3\lambda^1+\mbox{$\frac{\partial\lambda^1}{\partial s}$}\big)-\hat{A}_4(s,\sigma^1,\hat\gamma,\lambda^1,u^1,\tilde{x}^j).\notag
	\end{align}
\end{cor}

As in Corollary \ref{c1} optimal opinion of agent $1$ is a solution of a cubic equation $x^{1*}$ takes three values and because of rationality he chooses that $x^{1*}$ which has the maximum value. If $x^{1*}=\{x_1^{1*},x_2^{1*},x_3^{1*}\}$ then optimal opinion of the leader is $\bar x^{1*}=\max \{x_1^{1*},x_2^{1*},x_3^{1*}\}$. Under complete information all the other agents has the information about $\bar x^{1*}$ before a game starts and adjusts their opinions on it. The network is represented by a direct graph with edges directed from all the agents towards the leader. Thus $\eta_1=\emptyset$, $\eta_i=\{1\}, \forall i\in N\setminus \{1\}$ \citep{niazi2016}. Each of other agents represented by $i\in N\setminus\{1\}$ minimizes the expectation of his cost functional expressed in Equation (\ref{0}) where $w_{ij}\neq 0$ if $j=1$, subject to his stochastic opinion dynamics
\begin{align}\label{3}
dx^i(s)=\left[\mbox{$\frac{1}{\tilde{\lambda}_i}$}\left(k_ix^i(s)+w_{i1}\bar x^{1*}\right)+\hat\xi_i(s)\left(x^i(s)-\bar x^{1*}\right)-u^i(s)\right]ds+\sqrt{2\sigma}B^i(s),
\end{align}
where for all $i\in N\setminus\{1\}$, $\hat{\xi}_i(s)=\frac{w_{i1}\cosh\left(\sqrt{\tilde{\lambda}_i}(t-s)\right)}{\tilde{\lambda}_i\cosh\left(\sqrt{\tilde{\lambda}_i}t\right)}$, $\tilde{\lambda}_i=k_i+w_{i1}$, $u^i(s)$ is the control of opinion, $\sigma>0$ is a constant diffusion component and $B^i(s)$ the Brownian motion of agent $i$. In this framework we assume that, apart from the leader other agents have very small influence in $i^{th}$ agent's opinion. 

\begin{prop}
	Suppose, there is a network where all agents are unilaterally connected to their leader. Let agent $i$ minimizes his objective cost function
	\begin{align}\label{4}
	\E\left\{\mbox{$\frac{1}{2}$} \sum_{i=1}^{n-1}w_{i1}\left[x^i(s)-x^j(s)\right]^2+\mbox{$\frac{1}{2}$}k_i\left[x^i(s)-x_0^i\right]^2+\mbox{$\frac{1}{2}$}\left[u^i(s)\right]^2\right\},
	\end{align}
	subject to the stochastic opinion dynamics expressed in Equation (\ref{3}). For $b,d>0$, define $h^i(s,x^i)=\exp(sbx^i+d)$.
	
	(i) Then for 
	\begin{align}
	&f^i(s,\mathbf x,\lambda^i,u^i)\notag\\ &=\mbox{$\frac{1}{2}$}\sum_{i=1}^{n-1}w_{i1}\left(x^i-x^j\right)^2+\mbox{$\frac{1}{2}$}k_i\left(x^i-x_0^i\right)^2+\mbox{$\frac{1}{2}$}\left(u^i\right)^2+b\lambda^ix^ih^i(s,x^i)+\mbox{$\frac{\partial \lambda^i}{\partial s}$}h^i(s,x^i)\notag\\&\hspace{.5cm}+sb\lambda^ih^i(s,x^i)\left[\mbox{$\frac{1}{\tilde{\lambda}_i}$}\left(k_ix^i+w_{i1}\bar x^{1*}\right)+\hat\xi_i\left(x^i-\bar x^{1*}\right)-u^i\right]+s^2b^2\sigma\lambda^ih^i(s,x^i),\notag
	\end{align}
	we have a feedback Nash Equilibrium control of opinion dynamics
	\begin{align}\label{fc.0}
	\phi_0^{i*}(s,x^i)=\tilde p+\left\{\tilde q+\left[\tilde{q}^2+(\tilde r-\tilde{p}^2)^3\right]^{\frac{1}{2}}\right\}^{\frac{1}{3}}+\left\{\tilde q-\left[\tilde{q}^2+(\tilde r-\tilde{p}^2)^3\right]^{\frac{1}{2}}\right\}^{\frac{1}{3}},
	\end{align}
	where
	\begin{align}
	\tilde p&=-\frac{\tilde B_2(s,\hat\xi_i,x^i,x^j,\lambda^i)}{3B_1(s,x^i,\lambda^i)},\notag\\\tilde q&=\tilde{p}^3+\frac{\tilde B_2(s,\hat\xi_i,x^i,x^j,\lambda^i)\tilde B_3(s,\hat\xi_i,x^i,x^j,\lambda^i)-3\tilde B_1(s,x^i,\lambda^i)\tilde B_4(s,\hat\xi_i,x^i,x^j,\lambda^i)}{6[\tilde B_1(s,x^i,\lambda^i)]^2},\notag\\
	\tilde r &=\frac{\tilde B_3(s,\hat\xi_i,x^i,x^j,\lambda^i)}{3\tilde B_1(s,x^i,\lambda^i)},\notag
	\end{align}
	$\tilde B_1=(\tilde C_2)^2$, $\tilde B_2=-\tilde C_2(2\tilde A_2+\tilde C_1)$, $\tilde B_3=(\tilde A_2)^2-2\tilde A_2\tilde C_1\tilde C_2-(\tilde C_3)^2$, $\tilde B_4=\tilde A_1\tilde C_3-\tilde C_1(\tilde A_2)^2$, $\tilde C_1=sb\lambda^ih^i(s,x^i)$, $\tilde C_2=(sb)^3\lambda^ih^i(s,x^i)$, and $\tilde C_3=(sb)^2\lambda^ih^i(s,x^i)$.
	
	(ii) For a unique solution of the wave function $\Psi_{is}(x)$ as expressed in Proposition \ref{p1} and $\lambda^i(s)$ is a $C^2$ function with respect to $s$, an optimal opinion ${x}^{i*}$ is obtained by solving following equation
	\begin{align*}
	&sb^3\lambda^ih^i(s,x^i)(x^i)^2+h^i(s,x^i)\biggr\{2b\lambda^i+sb^2\mbox{$\frac{\partial\lambda^i}{\partial s}$}+s[1+sb^2\mbox{$\frac{\partial\lambda^i}{\partial s}$}+b\lambda^i(1+b+sb)]\left(\hat{\xi}_i+\frac{k_i}{\tilde{\lambda}_i}\right)\\&\hspace{.25cm}+s^2b^2\mbox{$\frac{\partial\hat{\xi}_i}{\partial s}$}+s^3b^4\sigma\lambda^i-sb^2\lambda^i\bigg[1+s\left(\hat{\xi}_i+\frac{k_i}{\tilde{\lambda}_i}\right)\bigg]\biggr\}x^i-(k_i+w_{i1})x^i+h^i(s,x^i)\times\\&\hspace{.5cm}\biggr\{sb\mbox{$\frac{\partial^2\lambda^i}{\partial s^2}$}+b\mbox{$\frac{\partial\lambda^i}{\partial s}$}-sb[sb\mbox{$\frac{\partial\lambda^i}{\partial s}$}+\lambda^i(1+b+sb)]\bigg[\biggr(\hat{\xi}_i+\frac{w_{i1}}{\tilde{\lambda}_i}\bigg)\bar x^{1*}+u^i\bigg]+\left(\hat{\xi}_i+\frac{k_i}{\tilde{\lambda}_i}\right)\times\\&\hspace{1cm}b(\lambda^i+\mbox{$\frac{\partial\lambda^i}{\partial s}$})+sb\lambda^i(1-sb\bar x^{1*})\mbox{$\frac{\partial\hat\xi_i}{\partial s}$}+s^2b^3\sigma\lambda^i(1+2s+s\mbox{$\frac{\partial\lambda^i}{\partial s}$})-b\tilde A_4(s,\sigma,\hat{\xi}_i,\lambda^i,u^i,x^j)\biggr\}\\&\hspace{1.25cm}+\tilde A_3(w_{i1},k_i,x^j)=0,
	\end{align*}
	which is
	\begin{align}\label{9.0}
	x^{i*}&=\tilde A_{12}+\left\{\tilde A_{13}+\left[(\tilde A_{13})^2+[\tilde A_{14}-(\tilde A_{12})^2]^3\right]^{\frac{1}{2}}\right\}^{\frac{1}{3}}+\left\{\tilde A_{13}-\left[(\tilde A_{13})^2+[\tilde A_{14}-(\tilde A_{12})^2]^3\right]^{\frac{1}{2}}\right\}^{\frac{1}{3}},
	\end{align}
	where 
	\begin{align}
	&\tilde A_{14}=\frac{\tilde A_{10}(s,\sigma,\hat\xi_i,\lambda^i,u^i)}{3\tilde A_8(s,\lambda^i)},\notag\\& \tilde A_{13}=(\tilde A_{12})^3+\frac{\tilde A_9(s,\sigma,\hat\xi_i,\lambda^i,u^i)\tilde A_{10}(s,\sigma,\hat\xi_i,\lambda^i,u^i)-3\tilde A_8(s,\lambda^i)\tilde A_{11}(s,\sigma,\hat\xi_i,w_{i1},k_i,\lambda^i,u^i,x^j)}{6[\tilde A_8(s,\lambda^i)]^2},\notag\\ & \tilde A_{12}=-\ \frac{\tilde A_9(s,\sigma,\hat\xi_i,\lambda^i,u^i)}{3\tilde A_8(s,\lambda^i)},\notag\\& \tilde A_{11}(s,\sigma,w_{i1},k_i,\hat\xi_i,\lambda^i,u^i,x^j)=\tilde A_3(w_{i1},k_i,x^j)+e\tilde A_7(s,\sigma,\hat\xi_i,\lambda^i,u^i),\notag\\ & \tilde A_{10}(s,\sigma,w_{i1},k_i,\hat\xi_i,\lambda^i,u^i)=k_i+w_{i1}+e\tilde A_6(s,\sigma,\hat\xi_i,\lambda^i,u^i)+sb\tilde A_7(s,\sigma,\hat\xi_i,\lambda^i,u^i),\notag\\ & \tilde A_9(s,\sigma,\hat\xi_i,\lambda^i,u^i)=e+\tilde A_5(s,\lambda^i)+sb\tilde A_6(s,\sigma,\hat\xi_i,\lambda^i,u^i),\notag\\& \tilde A_8(s,\lambda^i)=sb\tilde A_5(s,\lambda^i),\notag\\&\tilde A_7(s,\sigma,\hat{\xi}_i,\lambda^i,u^i)=sb\mbox{$\frac{\partial^2\lambda^i}{\partial s^2}$}+b\mbox{$\frac{\partial\lambda^i}{\partial s}$}-sb[sb\mbox{$\frac{\partial\lambda^i}{\partial s}$}+\lambda^i(1+b+sb)]\bigg[\biggr(\hat{\xi}_i+\frac{w_{i1}}{\tilde{\lambda}_i}\bigg)\bar x^{1*}+u^i\bigg]+\left(\hat{\xi}_i+\frac{k_i}{\tilde{\lambda}_i}\right)\notag\\&\hspace{1cm}\times b(\lambda^i+\mbox{$\frac{\partial\lambda^i}{\partial s}$})+sb\lambda^i(1-sb\bar x^{1*})\mbox{$\frac{\partial\hat\xi_i}{\partial s}$}+s^2b^3\sigma\lambda^i(1+2s+s\mbox{$\frac{\partial\lambda^i}{\partial s}$})-b\tilde A_4(s,\sigma,\hat{\xi}_i,\lambda^i,u^i),\notag\\ &\tilde A_6(s,\sigma,\hat{\xi}_i,\lambda^i,u^i)=2b\lambda^i+sb^2\mbox{$\frac{\partial\lambda^i}{\partial s}$}+s[1+sb^2\mbox{$\frac{\partial\lambda^i}{\partial s}$}+b\lambda^i(1+b+sb)]\left(\hat{\xi}_i+\frac{k_i}{\tilde{\lambda}_i}\right)\notag\\&\hspace{4cm}+s^2b^2\mbox{$\frac{\partial\hat{\xi}_i}{\partial s}$}+s^3b^4\sigma\lambda^i-sb^2\lambda^i\bigg[1+s\left(\hat{\xi}_i+\frac{k_i}{\tilde{\lambda}_i}\right)\bigg],\notag\\ &\tilde A_5(s,\lambda^i)=sb^3\lambda^i,\notag\\ &\tilde A_4(s,\sigma,\hat{\xi}_i,\lambda^i,u^i)=\lambda^i+s\mbox{$\frac{\partial\lambda^i}{\partial s}$}+s^2b\lambda^i\left(\frac{w_{i1}}{\tilde \lambda_i}\bar x^{1*}-\hat{\xi}_i\bar x^{1*}-u^i\right)+s\lambda^i\bigg[\hat{\xi}_i+\frac{k_i}{\tilde{\lambda}_i}+s^2b^3\sigma\bigg],\notag
	\end{align}
	and
	\begin{align}
	&\tilde A_3(w_{i1},k_i,x^j)=w_{i1}x^j+k_ix_0^i.\notag
	\end{align}
\end{prop}

\begin{proof}
	(i). For $b>0$, $d>0$ let $h^i(s,x^i)=\exp(sbx^i+d)$, $\frac{\partial}{\partial s}h^i(s,x^i)=bx^ih^i(s,x^i)$, $\frac{\partial}{\partial x^i}h^i(s,x^i)=sbh^i(s,x^i)$ and $\frac{\partial^2}{\partial (x^i)^2}h^i(s,x^i)=s^2b^2h^i(s,x^i)$. Hence, Proposition \ref{p0} implies,
	\begin{align}
	f^i(s,\mathbf x,\lambda^i,u^i)&=\mbox{$\frac{1}{2}$}\sum_{i=1}^{n-1}w_{i1}\left(x^i-x^j\right)^2+\mbox{$\frac{1}{2}$}k_i\left(x^i-x_0^i\right)^2+\mbox{$\frac{1}{2}$}\left(u^i\right)^2+b\lambda^ix^ih^i(s,x^i)+\mbox{$\frac{\partial \lambda^i}{\partial s}$}h^i(s,x^i)\notag\\&\hspace{.5cm}+sb\lambda^ih^i(s,x^i)\left[\mbox{$\frac{1}{\tilde{\lambda}_i}$}\left(k_ix^i+w_{i1}\bar x^{1*}\right)+\hat\xi_i\left(x^i-\bar x^{1*}\right)-u^i\right]+s^2b^2\sigma\lambda^ih^i(s,x^i).\notag
	\end{align}
	Now
	\begin{align}
	\mbox{$\frac{\partial}{\partial x^i}$}f^i(s,\mathbf x,\lambda^i,u^i)&=w_{i1}(x^i-x^j)+k_i(x^i-x_0^i)+bh^i(s,x^i)\biggr\{s\mbox{$\frac{\partial \lambda^i}{\partial s}$}\notag\\&\hspace{.5cm}+\lambda^i\bigg[1+bsx^i+s^2b\left[\mbox{$\frac{1}{\tilde{\lambda}_i}$}\left(k_ix^i+w_{i1}\bar x^{1*}\right)+\hat\xi_i\left(x^i-\bar x^{1*}\right)-u^i\right]\notag\\&\hspace{1cm}+s\left(\hat{\xi_i}+\frac{k_i}{\tilde{\lambda}_i}\right)+\sigma s^3b^2\bigg]\biggr\}\notag\\&=\tilde A_1(s,\hat{\xi}_i,x^i,x^j,\lambda^i)-s^2b^2\lambda^ih^i(s,x^i)u^i,\notag\\
	\mbox{$\frac{\partial^2}{\partial (x^i)^2}$}f^i(s,\mathbf x,\lambda^i,u^i)&=w_{i1}+k_i+sb^2h^i(s,x^i)\biggr\{s\mbox{$\frac{\partial \lambda^i}{\partial s}$}+\lambda^i\biggr[1+sbx^i+s\left(\hat{\xi}_i+\frac{k_i}{\tilde{\lambda}_i}\right)+\sigma s^3b^2\notag\\&\hspace{.5cm}+s^2b\left[\mbox{$\frac{1}{\tilde{\lambda}_i}$}\left(k_ix^i+w_{i1}\bar x^{1*}\right)+\hat\xi_i\left(x^i-\bar x^{1*}\right)-u^i\right]\biggr]\biggr\}+sb^2\lambda^i\left(\hat{\xi}_i+\frac{k_i}{\tilde{\lambda}_i}\right)h^i(s,x^i)\notag\\&=\tilde A_2(s,\hat\xi_i,x^i,x^j,\lambda^i)-s^3b^3\lambda^ih^i(s,x^i)u^i,\notag\\
	\mbox{$\frac{\partial}{\partial u^i}$}f^i(s,\mathbf x,\lambda^i,u^i)&=u^i-sb\lambda^ih^i(s,x^i),\notag
	\end{align}
	and,
	\begin{align}
	\mbox{$\frac{\partial^2}{\partial x^i\partial u^i}$}f^i(s,\mathbf x,\lambda^i,u^i)&=-s^2b^2\lambda^i h^i(s,x^i).\notag
	\end{align}
	Therefore, Equation (\ref{w21}) implies
	\begin{align}
	&\left[u^i-sb\lambda^ih^i(s,x^i)\right]\left[\tilde A_2(s,\hat\xi_i,x^i,x^j,\lambda^i)-s^3b^3u^i\lambda^ih^i(s,x^i)\right]^2\notag\\&=2s^2b^2\lambda^ih^i(s,x^i)\left[s^2b^2u^i\lambda^ih^i(s,x^i)-\tilde A_1(s,\hat\xi_i,x^i,x^j,\lambda^i)\right],\notag
	\end{align}
	and the cubic polynomial of agent $i$ with respect to control under the presence of a leader is 
	\begin{align}
	\tilde B_1(s,x^i,\lambda^i) (u^i)^3+\tilde B_2(s,\hat\xi_i,x^i,x^j,\lambda^i)(u^i)^2+\tilde B_3(s,\hat\xi_i,x^i,x^j,\lambda^i)u^i+\tilde B_4(s,\hat\xi_i,x^i,x^j,\lambda^i)=0,\notag
	\end{align}
	where $\tilde B_1=(\tilde C_2)^2$, $\tilde B_2=-\tilde C_2(2\tilde A_2+\tilde C_1)$, $\tilde B_3=(\tilde A_2)^2-2\tilde A_2\tilde C_1\tilde C_2-(\tilde C_3)^2$, $\tilde B_4=\tilde A_1\tilde C_3-\tilde C_1(\tilde A_2)^2$, $\tilde C_1(s,x^i,\lambda^i)=sb\lambda^ih^i(s,x^i)$, $\tilde C_2(s,x^i,\lambda^i)=(sb)^3\lambda^ih^i(s,x^i)$, and 
	$\tilde C_3(s,x^i,\lambda^i)=(sb)^2\lambda^ih^i(s,x^i)$.
	Therefore, feedback Nash equilibrium control under the presence of a leader is
	\begin{align}
	\phi_0^{i*}(s,x^i)=\tilde p+\left\{\tilde q+\left[\tilde{q}^2+(\tilde r-\tilde{p}^2)^3\right]^{\frac{1}{2}}\right\}^{\frac{1}{3}}+\left\{\tilde q-\left[\tilde{q}^2+(\tilde r-\tilde{p}^2)^3\right]^{\frac{1}{2}}\right\}^{\frac{1}{3}},\notag
	\end{align}
	where
	\begin{align}
	\tilde p&=-\frac{\tilde B_2(s,\hat\xi_i,x^i,x^j,\lambda^i)}{3B_1(s,x^i,\lambda^i)},\notag\\\tilde q&=\tilde{p}^3+\frac{\tilde B_2(s,\hat\xi_i,x^i,x^j,\lambda^i)\tilde B_3(s,\hat\xi_i,x^i,x^j,\lambda^i)-3\tilde B_1(s,x^i,\lambda^i)\tilde B_4(s,\hat\xi_i,x^i,x^j,\lambda^i)}{6[\tilde B_1(s,x^i,\lambda^i)]^2},\notag\\
	\tilde r &=\frac{\tilde B_3(s,\hat\xi_i,x^i,x^j,\lambda^i)}{3\tilde B_1(s,x^i,\lambda^i)}.\notag
	\end{align}
	
	(ii). Using Proposition \ref{p1} the right hand side of Equation (\ref{op0}) becomes,
	\begin{align}\label{5}
	& \mbox{$\frac{\partial}{\partial x^i}$}f^i[s,\mathbf x(s),\lambda^i(s),u^i(s)]\notag\\&=x^i(k_i+w_{i1})-(w_{i1}x^j+k_ix_0^i)+bh^i(s,x^i)\biggr\{sb\lambda^ix^i\bigg[1+s(\hat{\xi}_i+\frac{k_i}{\tilde{\lambda}_i})\bigg]+\lambda^i+s\mbox{$\frac{\partial\lambda^i}{\partial s}$}\notag\\&\hspace{1cm}+s^2b\lambda^i\left[\left(\frac{w_{i1}}{\tilde \lambda_i}-\hat{\xi}_i\right)\bar x^{1*}-u^i\right]+s\lambda^i\bigg[\hat{\xi}_i+\frac{k_i}{\tilde{\lambda}_i}+s^2b^3\sigma\bigg]\biggr\}\notag\\&=x^i(k_i+w_{i1})-\tilde A_3(w_{i1},k_i,x^j)+bh^i(s,x^i)\left[\tilde A_4(s,\sigma,\hat\xi_i,\lambda^i,u^i)+sb\lambda^ix^i\bigg[1+s\left(\hat{\xi}_i+\frac{k_i}{\tilde{\lambda}_i}\right)\bigg]\right],
	\end{align}
	the left hand side implies
	\begin{align}
	&\mbox{$\frac{\partial}{\partial s}$}f^i[s,\mathbf x(s),\lambda^i(s),u^i(s)]\notag\\&=h^i(s,x^i)\biggr\{\lambda^i(bx^i)^2+\mbox{$\frac{\partial^2\lambda^i}{\partial s^2}$}+bx^i\mbox{$\frac{\partial \lambda^i}{\partial s}$}\notag\\&\hspace{.25cm}+b\left[s\mbox{$\frac{\partial\lambda_i}{\partial s}$}+\lambda^i(1+sx^i)\right]\left[\frac{1}{\tilde{\lambda}_i}(k_ix^i+w_{i1}\bar x^{1*})+\hat{\xi}_i(x^i-\bar x^{1*})-u^i\right]\notag\\&\hspace{.5cm}+sb\lambda^i\left(x^i-\bar x^{1*}\right)\mbox{$\frac{\partial\hat\xi_i}{\partial s}$}+s^2b^2\sigma\mbox{$\frac{\partial\lambda^i}{\partial s}$}+sb^2\sigma\lambda_i[2+sbx^i]\biggr\},\notag
	\end{align}
	and
	\begin{align}\label{6}
	& \mbox{$\frac{\partial^2}{\partial s \partial x^i}$}f^i[s,\mathbf x(s),\lambda^i(s),u^i(s)]\notag\\&=h^i(s,x^i)\biggr\{2b\lambda^ix^i+sb^3\lambda^i(x^i)^2+sb\mbox{$\frac{\partial^2\lambda^i}{\partial s^2}$}+b(1+sbx^i)\mbox{$\frac{\partial\lambda^i}{\partial s}$}+\left[(sb)^2\mbox{$\frac{\partial\lambda^i}{\partial s}$}+sb\lambda^i(1+b+sb)\right]\times\notag\\&\hspace{1cm}\left[\frac{1}{\tilde{\lambda}_i}(k_ix^i-w_{i1}\bar x^{1*})+\hat{\xi}_i(x^i-\bar x^{1*})-u^i\right]+\left(\hat{\xi}_i+\frac{k_i}{\tilde{\lambda}_i}\right)\left[sb\mbox{$\frac{\partial\lambda^i}{\partial s}$}+b\lambda^i(1+sx^i)\right]\notag\\&\hspace{2cm}+sb\lambda^i\left[1+sb\left(x^i-\bar x^{1*}\right)\right]\mbox{$\frac{\partial\hat\xi_i}{\partial s}$}+s^2b^3\sigma\lambda^i[1+2s+sbx^i+s\mbox{$\frac{\partial\lambda^i}{\partial s}$}]\biggr\}.
	\end{align}
	Comparing Equations (\ref{5}) and \ref{6} we get,
	\begin{align}
	& h^i(s,x^i)\biggr\{2b\lambda^ix^i+sb^3\lambda^i(x^i)^2+sb\mbox{$\frac{\partial^2\lambda^i}{\partial s^2}$}+b(1+sbx^i)\mbox{$\frac{\partial\lambda^i}{\partial s}$}+\left[(sb)^2\mbox{$\frac{\partial\lambda^i}{\partial s}$}+sb\lambda^i(1+b+sb)\right]\times\notag\\&\hspace{1cm}\left[\frac{1}{\tilde{\lambda}_i}(k_ix^i-w_{i1}\bar x^{1*})+\hat{\xi}_i(x^i-\bar x^{1*})-u^i\right]+\left(\hat{\xi}_i+\frac{k_i}{\tilde{\lambda}_i}\right)\left[sb\mbox{$\frac{\partial\lambda^i}{\partial s}$}+b\lambda^i(1+sx^i)\right]\notag\\&\hspace{2cm}+sb\lambda^i\left[1+sb\left(x^i-\bar x^{1*}\right)\right]\mbox{$\frac{\partial\hat\xi_i}{\partial s}$}+s^2b^3\sigma\lambda^i[1+2s+sbx^i+s\mbox{$\frac{\partial\lambda^i}{\partial s}$}]\biggr\}\notag\\ &=x^i(k_i+w_{i1})-\tilde A_3(w_{i1},k_i,x^j)+bh^i(s,x^i)\left[\tilde A_4(s,\sigma,\hat\xi_i,\lambda^i,u^i)+sb\lambda^ix^i\bigg[1+s\left(\hat{\xi}_i+\frac{k_i}{\tilde{\lambda}_i}\right)\bigg]\right].\notag
	\end{align}
	The polynomial of the opinion dynamics is
	\begin{align}
	& sb^3\lambda^ih^i(s,x^i)(x^i)^2+h^i(s,x^i)\biggr\{2b\lambda^i+sb^2\mbox{$\frac{\partial\lambda^i}{\partial s}$}+s[1+sb^2\mbox{$\frac{\partial\lambda^i}{\partial s}$}+b\lambda^i(1+b+sb)]\left(\hat{\xi}_i+\frac{k_i}{\tilde{\lambda}_i}\right)\notag\\&\hspace{.25cm}+s^2b^2\mbox{$\frac{\partial\hat{\xi}_i}{\partial s}$}+s^3b^4\sigma\lambda^i-sb^2\lambda^i\bigg[1+s\left(\hat{\xi}_i+\frac{k_i}{\tilde{\lambda}_i}\right)\bigg]\biggr\}x^i-(k_i+w_{i1})x^i+h^i(s,x^i)\times\notag\\&\hspace{.5cm}\biggr\{sb\mbox{$\frac{\partial^2\lambda^i}{\partial s^2}$}+b\mbox{$\frac{\partial\lambda^i}{\partial s}$}-sb[sb\mbox{$\frac{\partial\lambda^i}{\partial s}$}+\lambda^i(1+b+sb)]\bigg[\biggr(\hat{\xi}_i+\frac{w_{i1}}{\tilde{\lambda}_i}\bigg)\bar x^{1*}+u^i\bigg]+\left(\hat{\xi}_i+\frac{k_i}{\tilde{\lambda}_i}\right)\times\notag\\&\hspace{1cm}b(\lambda^i+\mbox{$\frac{\partial\lambda^i}{\partial s}$})+sb\lambda^i(1-sb\bar x^{1*})\mbox{$\frac{\partial\hat\xi_i}{\partial s}$}+s^2b^3\sigma\lambda^i(1+2s+s\mbox{$\frac{\partial\lambda^i}{\partial s}$})-b\tilde A_4(s,\sigma,\hat{\xi}_i,\lambda^i,u^i)\biggr\}\notag\\&\hspace{1.25cm}+\tilde A_3(w_{i1},k_i,x^j)=0,\notag
	\end{align}
	or, 
	\begin{align}\label{7}
	& \tilde A_5(s,\lambda^i)h^i(s,x^i)(x^i)^2+\tilde A_6(s,\sigma,\hat\xi_i,\lambda^i,u^i)h^i(s,x^i)x^i+(k_i+w_{i1})x^i\notag\\&\hspace{.5cm}+\tilde A_7(s,\sigma,\hat\xi_i,\lambda^i,u^i)h^i(s,x^i)+\tilde A_3(w_{i1},k_i,x^j)=0.
	\end{align}
	As in Equation (\ref{7}) $h^i(s,x^i)=\exp(sbx^i+d)$, for $b>0$, $d>0$ fixed and a very small value of $x^i$ it can be approximated as $h^i(s,x^i)=1+(sbx^i+d)+o([sbx^i+d]^2)\approx1+d+sbx^i=e+sbx^i$ where assume $e=1+d$. 
	
	Therefore, we get a cubic equation expressed as,
	\begin{align}\label{8}
	& \tilde A_8(s,\lambda^i)(x^i)^3+\tilde A_9(s,\sigma,\hat\xi_i,\lambda^i,u^i)(x^i)^2\notag\\&+\tilde A_{10}(s,\sigma,w_{i1},k_i,\hat\xi_i,\lambda^i,u^i,x^j)x^i+\tilde A_{11}(s,\sigma,w_{i1},k_i,\hat\xi_i,\lambda^i,u^i,x^j)=0,
	\end{align}
	where
	\begin{align}
	& \tilde A_8(s,\lambda^i)=sb\tilde A_5(s,\lambda^i),\notag\\&\tilde A_9(s,\sigma,\hat\xi_i,\lambda^i,u^i)=e+\tilde A_5(s,\lambda^i)+sb\tilde A_6(s,\sigma,\hat\xi_i,\lambda^i,u^i),\notag\\&\tilde A_{10}(s,\sigma,w_{i1},k_i,\hat\xi_i,\lambda^i,u^i,x^j)=k_i+w_{i1}+e\tilde A_6(s,\sigma,\hat\xi_i,\lambda^i,u^i)+sb\tilde A_7(s,\sigma,\hat\xi_i,\lambda^i,u^i),\notag
	\end{align}
	and 
	\begin{align}
	\tilde A_{11}(s,\sigma,\hat\xi_i,\lambda^i,u^i,x^j)=\tilde A_3(w_{i1},k_i,x^j)+e\tilde A_7(s,\sigma,\hat\xi_i,\lambda^i,u^i).\notag
	\end{align}
	Solving Equation (\ref{8}) gives agent $i$'s optimal opinion
	\begin{align}\label{9}
	x^{i*}&=\tilde A_{12}+\left\{\tilde A_{13}+\left[(\tilde A_{13})^2+[\tilde A_{14}-(\tilde A_{12})^2]^3\right]^{\frac{1}{2}}\right\}^{\frac{1}{3}}+\left\{\tilde A_{13}-\left[(\tilde A_{13})^2+[\tilde A_{14}-(\tilde A_{12})^2]^3\right]^{\frac{1}{2}}\right\}^{\frac{1}{3}},
	\end{align}
	where 
	\begin{align}
	& \tilde A_{12}=-\ \frac{\tilde A_9(s,\sigma,\hat\xi_i,\lambda^i,u^i)}{3\tilde A_8(s,\lambda^i)},\notag\\& \tilde A_{13}=(\tilde A_{12})^3+\frac{\tilde A_9(s,\hat\xi_i,\gamma,\lambda^i,u^i)\tilde A_{10}(s,\sigma,w_{i1},k_i,\hat{\xi}_i,\lambda^i,u^i,x^j)-3\tilde A_8(s,\lambda^i)\tilde A_{11}(s,\sigma,w_{i1},k_i,\hat{\xi}_i,\lambda^i,u^i,x^j)}{6[\tilde A_8(s,\lambda^i)]^2},\notag
	\end{align}
	and
	\begin{align}
	\tilde A_{14}=\frac{\tilde A_{10}(s,\sigma,w_{i1},k_i,\hat{\xi}_i,\lambda^i,u^i,x^j)}{3\tilde A_8(s,\lambda^i)}.\notag
	\end{align}
\end{proof}

\section{Discussion}
This paper shows consensus as a feedback Nash equilibrium from a stochastic differential game. The same integral cost function has been used  as in \cite{niazi2016} subject to a stochastic opinion dynamics. A Feynman-type path integral approach has been used to construct a Wick-rotated Schr\"odinger type equation (i.e a Fokker-Plank diffusion equation). Finally, optimal opinion $x^{i*}$ and control $u^{i*}$ have been determined after solving the first order condition of the Wick-rotated Schr\"odinger equation. So far from my knowledge, this is a new approach. As different people have different opinions, an opinion changes over time and stubbornness and influence from others have some effects on individual decisions under the assumption that human body is a automaton. The fundamental assumption of this paper is opinion dynamics is stochastic in nature which is another contribution of this paper. Furthermore, results of this paper  give more generalized solution of opinion dynamics than \citep{niazi2016}.
\bibliography{bib}
\end{document}